\documentclass[a4paper]{amsart}

\usepackage{amsmath}
\usepackage{amssymb}
\usepackage{amsthm}
\usepackage{amscd}
\usepackage{bm}
\usepackage{bbm}
\usepackage{graphicx}
\usepackage{tikz,tikz-cd,pgf}
\usetikzlibrary{arrows.meta}
\usetikzlibrary{patterns}

\usepackage{ulem}
\usepackage{comment}

\usepackage{xcolor}
\usepackage{hyperref}

\newenvironment{solution}{\begin{proof}[Proof.]\parskip=0pt}{\end{proof}}

\newtheorem*{theorem}{Theorem}

\newtheorem{thm}{Theorem}[section]
\newtheorem{lemma}[thm]{Lemma}
\newtheorem{prop}[thm]{Proposition}

\theoremstyle{definition}

\newtheorem{expl}[thm]{Example}
\newtheorem{constr}[thm]{Construction}

\theoremstyle{remark}
\newtheorem*{rem}{Remark}

\numberwithin{equation}{section}

\def\le{\leqslant}
\def\ge{\geqslant}
\def\leq{\leqslant}
\def\geq{\geqslant}

\def\Ba{\mathrm{Ba}}
\def\Br{\mathrm{Br}}

\def\K{\mathcal{K}}
\def\H{\mathcal{H}}
\def\Z{\mathcal{Z}}
\def\C{\mathbb{C}}
\def\R{\mathbb{R}}

\def\ve{\varepsilon}
\def\D{\Delta}

\def\hra{\hookrightarrow}

\def\x{\times}
\def\ox{\otimes}

\def\sm{\setminus}
\def\ss{\subset}

\def\wt{\widetilde}
\def\wh{\widehat}

\def\q{\quad}
\def\mc{\mathcal}
\def\mb{\mathbb}

\def\oname{\operatorname}
\def\cone{\operatorname{cone}}

\def\relint{\operatorname{relint}}

\def\MF{\mathop\mathrm{MF}}

\newcommand{\6}{\partial}

\hypersetup{
    colorlinks=true,
    linkcolor=blue,
    urlcolor=blue,
    citecolor=red,
}
\makeatletter
    \def\@@and{and}
\makeatother

\title[Moment-angle manifolds and connected sums]{New families of moment-angle manifolds diffeomorphic to connected sums of products of spheres}

\author{Victoria Kovyrshina}
\address{Department of Mathematics and Mechanics, Moscow State University, Russia;\newline
Steklov Mathematical Institute of Russian Academy of Sciences, Moscow, Russia;\newline
National Research University Higher School of Economics, Moscow, Russia}
\email{potchtovy\_jashik@mail.ru}

\author{Taras Panov}
\address{Steklov Mathematical Institute of Russian Academy of Sciences, Moscow, Russia;\newline
Department of Mathematics and Mechanics, Moscow State University, Russia;\newline
Institute for Information Transmission Problems, Russian Academy of Sciences, Moscow, Russia;\newline
National Research University Higher School of Economics, Moscow, Russia}
\email{tpanov@mi-ras.ru}

\thanks{This work was performed within the framework of the state assignment at Steklov Mathematical Institute of Russian Academy of Sciences. Victoria Kovyrshina is supported by a stipend from the Theoretical Physics and Mathematics Advancement Foundation ``BASIS''}


\begin{document}

\begin{abstract}
We prove that the moment-angle manifold $\Z_\K$ is diffeomorphic to a connected sum of products of spheres when $\K$ is a starshaped (in particular, polytopal) $3$-dimensional simplicial sphere with exactly two missing edges that are not adjacent to each other. One of the summands of the connected sum is a product of three spheres.

For neighbourly starshaped simplicial spheres $\K$ of odd dimension, we prove the diffeomorphism $\Z_\K \cong M_1\#\cdots\# M_k$, where each $M_i$ is a product of two spheres.
We give an explicit description of the moment-angle manifolds corresponding to non-polytopal Barnette and Br\"uckner spheres. 
\end{abstract}

\maketitle

\section{Introduction}
A moment-angle complex is a topological space (a cell complex) with a torus action, studied in toric topology and the homotopy theory of polyhedral products~\cite{BP}. The topology of a moment-angle complex $\Z_\K$ is determined by the combinatorics of the corresponding simplicial complex~$\K$. If $\K$ is a simplicial triangulation of a sphere, then $\Z_\K$ is a topological manifold, called a moment-angle manifold. A moment-angle manifold has a smooth structure if $\K$ is a starshaped sphere (arises from a complete simplicial fan). In particular, $\Z_\K$ is smooth if $\K=\K_P$ is the dual complex of the boundary of a simple polytope $P$ (the nerve of the covering of the boundary of $P$ by facets); in this case the moment-angle manifold is traditionally denoted by~$\Z_P$.

There are several different geometric constructions of moment-angle manifolds that enrich their topology with remarkable geometric structures. One of them comes from holomorphic dynamics, where the moment-angle manifold $\Z_P$ arises as the leaf space of a holomorphic foliation on an open subset of a complex space. This leaf space is diffeomorphic to a nondegenerate intersection of Hermitian quadrics~\cite{bo-me06}, \cite[Chapter~6]{BP}. All early examples of moment-angle manifolds arising in this context were diffeomorphic to connected sums of products of spheres~\cite{lope89}. This is the case, for instance, when $P$ is a two-dimensional polytope (a polygon). However, from the description of the cohomology ring of $\Z_P$, it became clear that the topology of moment-angle manifolds is, in general, much more complicated than that of connected sums of products of spheres; for example, the ring $H^*(\Z_P)$ may have additive torsion of arbitrary order or nontrivial higher Massey products~\cite[Chapter~4]{BP}.

Nevertheless, it remains an open problem to describe the class of simple polytopes $P$ (or, more generally, simplicial spheres~$\K$) for which the moment-angle manifold $\Z_P$ is diffeomorphic (or, in a weaker formulation, homeomorphic or homotopy equivalent) to a connected sum of products of spheres. In the work of McGavran~\cite{mcga79}, it was proved that for two-dimensional polytopes $P$ (polygons), the manifold $\Z_P$ is PL homeomorphic to a connected sum of products of pairs of spheres; the same methods apparently also allow one to establish a diffeomorphism. 
Bosio and Meersseman observed that McGavran's result extends to a wider family of polytopes: according to~\cite[Theorem~6.3]{bo-me06}, the moment-angle manifold $\Z_P$ is homeomorphic to a connected sum of products of spheres if $P$ is dual to a \emph{stacked} polytope, i.\,e. is obtained from a simplex by  iterating the vertex cut operation. Moreover, in~\cite[Proposition~11.6]{bo-me06} it is proved that for three-dimensional polytopes $P$ (or two-dimensional simplicial spheres $\K$) the manifold $\Z_P$ is diffeomorphic to a connected sum of products of spheres if and only if $P$ is a cube or is obtained from the 3-simplex by iterating the vertex cut operation. This description can be complemented by two equivalent conditions: chordality and minimal non-Golodness (see~\cite[Proposition~3.1]{kov-pa}).

In the work of Gitler and L\'opez de Medrano~\cite{gi-lo13}, a technique was developed to bear on the study of moment-angle manifolds $\Z_P$ given by intersections of quadrics in the smooth category. This technique, based on the use of simply connected surgeries, representation of homology classes by embedded spheres with trivial normal bundles and the $h$-cobordism theorem, led to proving the diffeomorphism of the moment-angle manifold $\Z_P$ and a connected sum of products of pairs of spheres for dual-neighbourly polytopes $P$ of even dimension (neighbourly polytopal spheres $\K_P$ of odd dimension) and for dual-stacked polytopes $P$ under the restriction $m<3n$ on the number of facets. This restriction was later removed in the work of Chen, Fan and Wang~\cite{c-f-w20}; thus, it was proved that McGavran's result~\cite{mcga79} on moment-angle manifolds corresponding to polygons and its generalization to dual-stacked polytopes hold in the smooth category as well.

A fundamentally new example was found in the work of Fan, Chen, Ma and Wang~\cite{FCMW}, where a four-dimensional polytope $P$ was constructed for which the moment-angle manifold $\Z_P$ has the cohomology ring isomorphic to the cohomology ring of a connected sum of products of spheres where one of the summands is a product of \emph{three} spheres. Subsequently, Iriye~\cite{iriy18} proved that there is in fact a diffeomorphism between $\Z_P$ and the corresponding connected sum of products of spheres.

In~\cite[Theorem~4.4]{kov-pa}, the authors gave a complete description of the class of three-dimensional simplicial spheres $\K$ for which the cohomology ring of the moment-angle manifold $\Z_\K$ is isomorphic to the cohomology ring of a connected sum of products of spheres:

\begin{theorem}[Theorem~\ref{chordal4iff}]
Let $\K$ be a 3-dimensional simplicial sphere. There is a ring isomorphism $H^*(\Z_\K) \cong H^*(M_1\#\cdots\# M_k)$ where each $M_i$ is a product of spheres, if and only if one of the following conditions is satisfied:
\begin{itemize}
\item[(a)] $\K=S^0*S^0*S^0*S^0$ (the boundary of a $4$-dimensional cross-polytope);
\item[(b)] $\K^1$ is a chordal graph;
\item[(c)] $\K^1$ has exactly two missing edges and they are not adjacent to each other (i.\,e. form a chordless $4$-cycle).
\end{itemize}
\end{theorem}

In case~(c), we have $H^*(\Z_\K) \cong H^*(M_1\#\cdots\# M_k)$, where one of the summands $M_i$ is a product of \emph{three} spheres. The example from~\cite{FCMW} falls into this category. In~\cite{kovy}, a large family of four-dimensional simple polytopes $P$ was constructed for which the dual sphere $\K_P$ satisfies condition~(c) of the above theorem.

In the present work, we show that under condition~(c) the isomorphism of cohomology rings of the above theorem comes from a diffeomorphism provided that the sphere $\K$ is starshaped (in particular, polytopal). As noted above, this is precisely the condition that ensures the existence of a smooth structure on $\Z_\K$:

\begin{theorem}[Theorems~\ref{diffeo4}, \ref{diffeo4-starshaped}]
Let $\K$ be a 3-dimensional starshaped simplicial sphere with only two missing edges which are not adjacent to each other. Then there is a diffeomorphism 
\[
  \Z_\K \cong M_1\#\cdots\# M_k,
\]
where $M_1\cong S^3\times S^3\times S^{m-2}$, and each $M_i$ with $i\ge 2$ is a product of two spheres.
\end{theorem}

In case~(b) of Theorem~\ref{chordal4iff}, we prove the existence of a diffeomorphism for starshaped spheres $\K$ under an additional connectedness assumption:

\begin{theorem}[Theorem~\ref{connected-starshaped}]
Let $\K$ be a 3-dimensional starshaped sphere such that $\K^1$ is a chordal graph and there is a vertex adjacent to all other vertices. Then there is a diffeomorphism $\Z_\K \cong M_1\#\cdots\# M_k$, where each $M_i$ is a product of two spheres.
\end{theorem}

In particular, $\Z_\K$ is diffeomorphic to a connected sum of products of spheres if $\Z_\K$ is a starshaped 3-dimensional \emph{neighbourly} sphere, i.\,e. if $\K^1$ is a complete graph.

A simplicial $(n-1)$-sphere $\K$ is called \emph{neighbourly} if every set of $\lfloor\frac{n}{2}\rfloor$ vertices forms a face of~$\K$. 
In~\cite[\S11]{bo-me06}, a conjecture was formulated that for neighbourly polytopal spheres $\K_P$ the manifold $\Z_P$ is a connected sum of products of spheres. In the case of polytopal spheres $\K_P$ of odd dimension (i.\,e. polytopes $P$ of even dimension), this conjecture was proved in~\cite[Theorem~1.3]{gi-lo13}. 
As shown in the work of Amelotte and Briggs, odd-dimensional neighbourly spheres $\K$ are characterized homologically by the condition that the Stanley--Reisner ideal $I_\K$ admits an almost linear resolution~\cite[Proposition~5.3]{am-br}; moreover, the corresponding manifold $\Z_\K$ is rationally homotopy equivalent to a connected sum of products of spheres~\cite[Theorem~5.5]{am-br}. In the work of Membrillo Solis and Theriault~\cite{me-th}, it was proved by homotopy-theoretic methods that for a neighbourly sphere $\K$ of odd dimension the manifold $\Z_\K$ is homotopy equivalent to a connected sum of products of spheres. We prove that for odd-dimensional neighbourly starshaped spheres there is actually a diffeomorphism:

\begin{theorem}[Theorem~\ref{odd-dim-connected-starshaped}]
Let $\K$ be a starshaped neighbourly simplicial sphere of odd dimension $n-1$ with $\K\ne\partial\Delta^n$. Then there is a diffeomorphism $\Z_\K \cong M_1\#\cdots\# M_k$, where each $M_i$ is a product of two spheres.
\end{theorem}

The first examples of 3-dimensional simplicial spheres not arising from boundaries of convex polytopes are the \emph{Barnette sphere} $\K_\Ba$ and the \emph{Br\"uckner sphere} $\K_\Br$. Each of them has 8 vertices and is starshaped, with $\K_\Br$ being neighbourly (any two vertices are joined by an edge), while $\K_\Ba$ has exactly one missing edge. Our results imply that the moment-angle manifolds corresponding to the Barnette and Br\"uckner spheres are diffeomorphic to the following connected sums of products of spheres (Proposition~\ref{Barnette-Bruckner}):
\begin{align*}
    &\Z_{\K_\Ba} \cong S^3 \x S^9 \# (S^5 \x S^7)^{\#12} \# (S^6 \x S^6)^{\#12}\\
    &\Z_{\K_\Br} \cong (S^5 \x S^7)^{\#16} \# (S^6 \x S^6)^{\#15} \,.
\end{align*}

Our proof of the diffeomorphism between the moment-angle manifold $\Z_\K$ and a connected sum of products of spheres is based on an explicit representation of the generators of homology groups by submanifolds \emph{with trivial normal bundles}. In the case when $\K$ is a polytopal sphere, the trivialization of the normal bundles comes from the representation of the moment-angle manifold $\Z_P$ and its submanifolds corresponding to faces and missing faces as nondegenerate intersections of quadrics (Propositions~\ref{ftriv} and~\ref{qtriv}, Lemmas~\ref{normal-mf-submanifold} and~\ref{normal_bundles}). In the more general case of starshaped spheres $\K$ (arising from complete simplicial fans), the triviality of the normal bundles is established using the quotient construction by a proper exponential action (Lemmas~\ref{normal-mf-starshaped} and~\ref{normal_bundles-starshaped}).

\section{Preliminaries and notation}

Let $\K$ be a simplicial complex on the set $[m]=\{1, \ldots, m\}$. We assume that $\K$ contains an empty set $\varnothing$ and all one element subsets $\{i\} \ss [m]$ (i.\,e. $\K$ has no \emph{ghost} vertices), unless otherwise specified. The dimension of $\K$ is the maximal dimension of its simplexes.

We denote the full subcomplex of $\K$ on a vertex set $J = \{j_1, \ldots,j_k\} \ss [m]$ by $\K_J$ or by $\K_{\{j_1, \ldots, j_k\}}$.

The \emph{moment-angle complex} $\Z_\K$ corresponding to $\K$ is defined as follows (see \cite[\S4.1]{BP}):
\[
 \Z_\K = \underset{I \in \K}{\bigcup} \Bigl( \underset{i \in I}{\prod} D^2 \x \underset{i \notin I}{\prod} S^1 \Bigr)
 \subset\prod_{i=1}^m D^2\,.
\]

\begin{lemma}\label{retract}
Let $\K_J$ be the full subcomplex of a simplicial complex $\K$ on $[m]$. Then $\Z_{\K_J}$ is a retract of $\Z_\K$ and $H^*(\Z_{\K_J})$ is a subring of $H^*(\Z_\K)$.
\end{lemma}
\begin{solution}
Consider the canonical inclusion $i\colon \Z_\K \hra (D^2)^m$, and let $q\colon (D^2)^m \rightarrow (D^2)^{|J|}$ be the map that omits the coordinates corresponding to $[m] \sm J$. Then $r = q \circ i\colon \Z_\K \rightarrow \Z_{\K_J}$ is the required retraction, and it induces an injective homomorphism $H^*(\Z_{\K_J})\to H^*(\Z_\K)$ in cohomology.
\end{solution}

\begin{thm}[{\cite[Theorem 4.5.8]{BP}}]\label{Hochster}
There are isomorphisms of groups
\[
H^l(\Z_\K) \cong \underset{J \ss [m]}{\bigoplus} \wt{H}^{l - |J| - 1}(\K_J) 
\]
These isomorphisms combine to form a ring isomorphism 
\[
H^*(\Z_\K) \cong \underset{J \ss [m]}{\bigoplus} \wt{H}^*(\K_J),
\]
where the ring structure on the right hand side is given by the canonical maps
\[
H^{k - |I| - 1}(\K_I)  \ox H^{l - |J| - 1}(\K_J) \longrightarrow H^{k + l - |I| - |J| - 1}(\K_{I \cup J}) \,,
\]
which are induced by the simplicial maps $\K_{I \cup J} \rightarrow \K_I * \K_J$ for $I \cap J = \varnothing$ and zero otherwise.
\end{thm}

We use the notation
\[
  \H^{l, J} = \wt{H}^l(\K_J),\quad 
  \H^{*, J} = \wt{H}^*(\K_J)\quad\text{and}\quad
  \H^{l, *}(\K) = \underset{J \ss [m]}{\bigoplus} \wt{H}^l(\K_J)
\]  
and similarly for homology. The ring structure in $H^*(\Z_\K)=\H^{*,*}(\K)$ is given by the maps
\begin{equation}\label{multH}
  \H^{k,I}\ox\H^{l,J}\longrightarrow \H^{k+l+1,I\sqcup J},\qquad k,l\ge0,\; I\cap J=\varnothing.
\end{equation}

A (convex) \emph{polytope} $P$ is a bounded intersection of a finite number of halfspaces in a real affine space. A \textit{facet} of a polytope $P$ is its face of codimension~$1$.

A polytope $P$ of dimension $n$ is called \textit{simple} if each vertex of $P$ belongs to exactly $n$ facets. So if $P$ is simple, then the dual polytope $P^*$ is simplicial and its boundary $\6 P^*$ is a simplicial complex, which we denote by $\K_P$. Then $\K_P$ is the nerve complex of the covering of $\6 P$ by its facets. The moment-angle complex   $\Z_{\K_P}$ corresponding to $\K_P$ is denoted simply by~$\Z_P$.

The facets of $P$ correspond to the vertices of $\K_P$, so we denote by $F_i$ the facet of $P$ dual to the vertex $\{i\} \in \K_P$, and for $I = \{i_1, \ldots, i_k\} \ss [m]$ we denote $F_I = F_{i_1} \cap \ldots \cap F_{i_k}$. Note that $F_I$ is a nonempty face of $P$ if and only if $I \in \K_P$.

A \emph{simplicial sphere} (or \emph{triangulated sphere}) is a simplicial complex $\K$, whose geometric realisation is homeomorphic to a sphere. If $P$ is a simple polytope of dimension~$n$, then the nerve complex $\K_P$ is a simplicial sphere of dimension~$n-1$.
For $n\le3$, any simplicial sphere of dimension $n-1$ is combinatorially equivalent to the nerve complex $\K_P$ of a simple $n$-dimensional polytope~$P$. This is not true for $n\ge4$; the \emph{Barnette sphere} is a famous example of a $3$-dimensional simplicial sphere with $8$ vertices that is not combinatorially equivalent to the boundary of a convex $4$-dimensional polytope (see Example~\ref{bbexa} below).

\begin{thm}[{\cite[Theorem 4.1.4, Corollary 6.2.5]{BP}}]\label{manifold}
Let $\K$ be a simplicial sphere of dimension $n - 1$ with $m$ vertices. Then $\Z_\K$ is a closed topological manifold of dimension $m+n$.
If $P$ is a simple $n$-dimensional polytope with $m$ facets, then $\Z_P$ is a smooth manifold of dimension $m+n$.
\end{thm}

Cohomology of moment-angle-manifolds $\Z_\K$ corresponding to simplicial spheres $\K$ (or, more generally, \emph{Gorenstein}${}^*$ complexes~$\K$) satisfies the multigraded Poincar\'e duality, which reduces to Alexander duality for subcomplexes of~$\K$~\cite[\S3.4, \S4.6]{BP} due to isomorphisms from Theorem~\ref{Hochster}. This means that we have isomorphisms
\begin{equation}\label{aldua}
  \H^{l,J}\cong \H_{n-2-l,[m]\setminus J},
\end{equation}
and also nontrivial multiplications
\[
  \H^{l,J}\otimes\H^{n-2-l,[m]\setminus J}\longrightarrow
  \H^{n-1,[m]}\cong\mathbb Z.
\]

A simplicial polytope $Q$ is called \textit{stacked} if it can be obtained from a simplex by a sequence of stellar subdivisions of facets. Equivalently, the dual simple polytope $P=Q^*$ is obtained from a simplex by iterating the vertex cut operation.

A \emph{connected sum of products of spheres} is a closed smooth $n$-dimensional manifold $M$ diffeomorphic to a connected sum $M_1\#\cdots\# M_k$ where each $M_k$ is a product of spheres $S^{n_{k1}}\x\cdots\x S^{n_{kl}}$, $n_{k1}+\cdots+n_{kl}=n$.

The following theorem was proved in~\cite[Theorem~2.2]{gi-lo13} for $m<3n$ and in~\cite{c-f-w20} in the general case. In the topological category, this theorem follows from the results~\cite{mcga79}; see~\cite[Theorem~6.3]{bo-me06}.

\begin{thm}[{\cite[Theorem~1.3]{c-f-w20}}]\label{T4.6.12}
Let $P$ be a dual stacked $n$-dimensional polytope with $m > n + 1$ facets. Then the corresponding moment-angle manifold is diffeomorphic to a connected sum of products of pairs of spheres, namely,
\[
  \Z_P \cong \underset{k=3}{\overset{m-n+1}{\#}} (S^k \x S^{m+n-k})^{\#(k-2)\binom{m-n}{k-1}}.
\]
\end{thm}

In particular, the moment-angle complex corresponding to a polygon (a two-dimensional polytope) is a connected sum of products of pairs of spheres.

A \emph{graph} $\Gamma$ is a one-dimensional simplicial complex.
A graph $\Gamma$ is called \emph{chordal} if every cycle of $\Gamma$ with more than $3$ vertices has a chord, where a chord is an edge connecting two vertices that are not adjacent in the cycle.
The vertices of a graph are in \emph{perfect elimination order} if for any vertex $\{i\}$ all its neighbours with indices less than $i$ are pairwise adjacent. According to the theorem~\cite{FG}, a graph is chordal if and only if its vertices can be arranged in a perfect elimination order.

\section{Intersections of quadrics and submanifolds with trivial normal bundles}\label{seciq}

The moment-angle-manifold $\Z_P$ corresponding to a simple $n$-dimensional polytope $P$ with $m$ facets embeds into the complex space $\C^m$ as a nondegenerate intersection of $m-n$ hermitian quadratic hypersurfaces, and therefore has a trivial normal bundle. Moreover, facets of $P$ and also intersections of $P$ with transverse planes determine submanifolds in $\Z_P$ with trivial normal bundles. In this section we provide the corresponding constructions, which will be useful for representing homology classes in the proof of the diffeomorphism of $\Z_P$ to a connected sum of products of spheres.

Let $W$ be a real vector space of dimension~$n$. A configuration $\mathrm A=\{a_1,\ldots,a_m\}$ of $m$ vectors in the dual space $W^*$ and a set of numbers $b=(b_1,\ldots,b_m) \in \R^m$ give a system of linear inequalities (a configuration of halfspaces)
\[
  (\mathrm A,b)=\{\langle a_i,w\rangle+b_i\ge0,\;
  i=1,\ldots,m\}
\]
that defines a convex \emph{polyhedron}
\begin{equation}\label{ptope}
  P=P(\mathrm A,b)=\{w\in W\colon\langle a_i,w\rangle+b_i\ge0,\;
  i=1,\ldots,m\}.
\end{equation}
A \emph{polytope} is a bounded polyhedron.

We assume that the vectors of configuration $\mathrm A$ span the whole $W^*$; this implies $m\ge n$. Then if the polyhedron $P(\mathrm A,b)$ is nonempty, it has at least one vertex. (Indeed, otherwise the minimal face of $P(\mathrm A,b)$ is a subspace of positive dimension on which all vectors from $\mathrm A$ vanish.)

The inequality $\langle a_i,w\rangle+b_i\ge0$ from the system $(\mathrm A,b)$ is called \emph{redundant} if removing it does not change the polyhedron $P(\mathrm A,b)$, i.\,e. $P(\mathrm A,b)=P(\mathrm A\setminus a_i,b\setminus b_i)$. The configuration $\mathrm A$ may contain zero vectors; if $a_i=0$ and $b_i\ge0$, then the corresponding inequality $\langle a_i,w\rangle+b_i\ge0$ is redundant.

We say that the system $(\mathrm A,b)$ is \emph{generic} if $\mathrm A$ spans $W^*$, the polyhedron $P(\mathrm A,b)$ is nonempty and each point $w\in P(\mathrm A,b)$ belongs to at most $n$ hyperplanes $\langle a_i,w\rangle+b_i=0$, $i=1,\ldots,m$.

\begin{prop}
Let $P=P(\mathrm A,b)$ be a polyhedron given by a generic system. Then
\begin{itemize}
\item[(1)]
  $P$ is simple and has full dimension $n=\dim W$;
\item[(2)]
  if the inequality $\langle a_i,w\rangle+b_i\ge0$ is redundant in the system $(\mathrm A,b)$, then $\langle a_i,w\rangle+b_i>0$ for all $w\in P$ (i.\,e. the hyperplane corresponding to the redundant inequality does not intersect~$P$).
\end{itemize}
\end{prop}

\begin{proof}
To prove (1), take any vertex $w\in P$. It belongs to exactly $n$ facets (to at most $n$ since the system is generic, and to at least $n$ since $\dim W=n$). Without loss of generality assume that $\langle a_i,w\rangle+b_i=0$ for $i=1,\ldots,n$ and $\langle a_i,w\rangle+b_i>0$ for $i>n$, so $a_1,\ldots,a_n$ is a basis of $W^*$. Let $w_1,\ldots,w_n$ be the dual basis of~$W$. Define $w'=w+\varepsilon_1w_1+\cdots+\varepsilon_n w_n$ where $\varepsilon_i>0$. Then we have
\[
  \langle a_i,w'\rangle+b_i=\langle a_i,w\rangle+b_i+
  \varepsilon_1\langle a_i,w_1\rangle+\cdots+\varepsilon_n\langle a_i,w_n\rangle,
\]
which is positive for every $i=1,\ldots,m$ if $\varepsilon_i$ are sufficiently small. Hence $w'$ belongs to the interior of the polyhedron $P(\mathrm A,b)$ and therefore it has full dimension.

\smallskip

To prove (2), suppose that the hyperplane $\langle a_i,w\rangle+b_i=0$ corresponding to a redundant inequality intersects $P$. Then it intersects $P$ along some face. This face contains a vertex, which is also a vertex of the polyhedron~$P$ and is contained in its $n$ facets by claim~(1). Together with the redundant hyperplane $\langle a_i,w\rangle+b_i=0$, we obtain $n+1$ hyperplanes containing this vertex. A contradiction.
\end{proof}

\begin{constr}[Gale duality]\label{galed}
For a configuration $\mathrm A=\{a_1,\ldots,a_m\}$ spanning the linear space   $W^*\cong\R^n$, consider the linear map $A\colon\R^m\to W^*$ given by $e_i\mapsto a_i$, where $e_i$ is the $i$-th vector of the standard basis in~$\R^m$. The map $A$ is surjective by assumption, we complete it to the exact sequence
\[
  0\longrightarrow V\longrightarrow\R^m\stackrel{A}\longrightarrow W^*  \longrightarrow 0,
\]
and consider the dual exact sequence
\[
  0\longrightarrow W\stackrel{A^*}\longrightarrow\R^m
  \stackrel\varGamma\longrightarrow
  V^*\longrightarrow 0,
\]
where $A^*(w)=(\langle a_1,w\rangle,\ldots,\langle a_m,w\rangle)$.

Define $\gamma_i=\varGamma(e_i)$, then $\Gamma=\{\gamma_1,\ldots,\gamma_m\}$ is a spanning configuration of $m$ vectors in the space~$V^*\cong\R^{m-n}$. The configuration $\Gamma$ is called the \emph{Gale dual} to $\mathrm A=\{a_1,\ldots,a_m\}$. If we choose bases in $W^*$ and $V^*$, then the maps $A$ and $\varGamma$ are given by matrices of sizes $n\times m$ and $(m-n)\times m$ respectively, and the relation $\varGamma A^*=0$ expresses the fact that the rows of the matrix $\varGamma$ form a basis in the space of linear relations among the vectors $a_1,\ldots,a_m$. The following standard property of Gale dual configurations holds.
\end{constr}

\begin{prop}\label{dspan}
For any $I\subset[m]$ the subconfiguration $\mathrm A_I = \{a_i\colon i\in I\}$ is linearly independent in $W^*$ if and only if the subconfiguration $\Gamma_{\widehat I}=\{\gamma_j\colon j\notin I\}$ spans~$V^*$.
\end{prop}

\begin{constr}[$\Z_P$ as the intersection of quadrics]
For a polyhedron~\eqref{ptope} given by a system of linear inequalities, consider the injective affine map
\[
  i_{\mathrm A,b}\colon W\to\R^m,\quad
  w\mapsto A^*w+b=(\langle a_1,w\rangle+b_1,\ldots,\langle a_m,w\rangle+b_m).
\]
The image $i_{\mathrm A,b}(P)$ is the intersection of the nonnegative orthant
\[
  \R^m_\ge=\{y=(y_1,\ldots,y_m)\in\R^m\colon y_i\ge0,\;i=1,\ldots,m\}
\]
with the $n$-dimensional plane
\begin{equation}\label{iAb}
\begin{aligned}
  i_{\mathrm A,b}(W) &= \{y\in\mathbb{R}^{m}\colon y=A^{*}w+b\quad\text{for some }w\in W\} \\
  &= \{y\in\mathbb{R}^{m}\colon \varGamma y=\varGamma b\}.
\end{aligned}
\end{equation}

Consider the standard moment map
\[
  \mu\colon\C^m\to\R^m,\quad(z_1,\ldots,z_m)\mapsto
  (|z_1|^2,\ldots,|z_m|^2)
\]
and define the space $\Z_{\mathrm A,b}$ from the pullback square
\[
\begin{tikzcd}[row sep=small]
  \Z_{\mathrm A,b} \ar{r}{i_{\Z}} \ar{d} & \C^m \ar{d}{\mu}\\
  P \ar{r}{i_{\mathrm A,b}} & \R_{\geq}^m
\end{tikzcd}
\]
Replacing $y_k$ by $|z_k|^2$ in the equations defining the affine plane~\eqref{iAb} we obtain that $\Z_{\mathrm A,b}$ embeds into $\C^m$ as the set of common zeros of $m-n$ quadratic equations (\textit{hermitian quadrics}):
\begin{equation}\label{intqu}
  i_\Z(\Z_{\mathrm A,b}) =\mu^{-1}(i_{\mathrm A,b}(P))= \bigl\{ z \in \C^m \colon
  \gamma_1|z_1|^2+\cdots+\gamma_m|z_m|^2=\delta\bigr\},
\end{equation}
where $\delta=\varGamma b=b_1\gamma_1+\cdots+b_m\gamma_m\in V^*$.
The torus $T^m$ acts on $\Z_{\mathrm A,b}$ with quotient $P$ and $i_\Z$ is a $T^m$-equivariant embedding.
\end{constr}

\begin{thm}[{\cite[Theorems~6.1.5, 6.2.4]{BP}}]\label{zpqua}\
\begin{itemize}
\item[(1)] The intersection of quadrics~\eqref{intqu} is nonempty and nondegenerate if and only if the system $(\mathrm A,b)$ is generic.

\item[(2)] If, in addition to the conditions of~\emph{(1)}, the polyhedron $P=P(\mathrm A,b)$ is bounded (i.\,e., is a simple polytope), then the intersection of quadrics~\eqref{intqu} is $T^m$-equi\-va\-ri\-ant\-ly homeomorphic to the product $\mathcal Z_P\times T^r$, where $r$ is the number of redundant inequalities in the system $(\mathrm A,b)$.
\end{itemize}
Thus, the moment-angle manifold $\Z_P$ embeds as a smooth real submanifold into $\mathbb C^m$ with trivial normal bundle.
\end{thm}

Note that the intersection of quadrics~\eqref{intqu} does not change under a shift of $P$ by an arbitrary vector $w'\in W$. Indeed, the polyhedron $P-w'$ is given by the system of inequalities $A^*w+(A^*w'+b)\ge0$, i.\,e. $P-w'=P(\mathrm A,A^*w'+b)$, and for the corresponding intersection of quadrics~\eqref{intqu} we have $\delta'=\varGamma(A^*w'+b)=\varGamma b=\delta$.

\begin{prop}\label{ftriv}
Let $G$ be a face of a simple polytope $P=P(\mathrm A,b)$ given by a generic system of inequalities~\eqref{ptope}. Then $\Z_G$ embeds as a smooth submanifold in $\Z_P$ with trivial normal bundle.
\end{prop}

\begin{proof}
The face $G$ is given by setting some of the inequalities in \eqref{ptope} to equalities. We may assume $0\in G$; this can always be achieved by shifting the polytope $P$. Then for some 
$I=\{i_1,\ldots,i_k\}\subset[m]$ we have
\[
  G=\{w\in P\colon \langle a_i,w\rangle=0
  \text{ for }i\in I\}.
\]
Therefore $G$ is given by the generic system of inequalities
\[
  (\mathrm A_{\widehat I},b_{\widehat I})=\{\langle a_j,w\rangle+b_j\ge0,\; j\notin I\}
\]  
in the space $W_{\widehat I}=\{w\in P\colon \langle a_i,w\rangle=0  \text{ for }i\in I\}$ of dimension $n-|I|$. 

Since the vectors $\mathrm A_I$ are linearly independent, the Gale dual configuration to $\mathrm A_{\widehat I}$ in $W_{\widehat I}$ is $\Gamma_{\widehat I}$ in the same space $V^*$. The intersection of quadrics~\eqref{intqu}, corresponding to the polyhedron $G=P(\mathrm A_{\widehat I},b_{\widehat I})$ is
\begin{equation}\label{Gquad}
  \bigl\{z\in\C^{\widehat I}\colon \sum_{j\in\widehat I}\gamma_j|z_j|^2=\delta\bigr\},
\end{equation}
where $\C^{\widehat I}=\{z\in\C^m\colon z_i=0\text{ for }i\in I\}$ and $\delta=\sum_{j\notin I}b_j\gamma_j$ since $b_i=0$ for $i\in I$.
By Theorem~\ref{zpqua} the nondegenerate intersection of quadrics~\eqref{Gquad} is homeomorphic to $\mathcal Z_G\times T^{r_G}$, where $r_G$ is the number of redundant inequalities in the system~$(\mathrm A_{\widehat I},b_{\widehat I})$. 

On the other hand, the system of quadrics~\eqref{Gquad} is obtained from~\eqref{intqu} by adding the conditions $z_i=0$ for $i\in I$. Therefore, the system
\[
  \{\gamma_1|z_1|^2+\cdots+\gamma_m|z_m|^2=\delta\}\cup
  \{z_i=0,\; i\in I\}
\]
of $m-n+|I|$ hypersurfaces in $\C^m$ defining $\mathcal Z_G\times T^{r_G}$ is nondegenerate. Hence each of the embeddings in the diagram
\[
\begin{tikzcd}[row sep=small]
  \Z_G\times T^{r_G}\cong
  \Z_{\mathrm A_{\widehat I},b_{\widehat I}} 
  \ar[hook]{r} \ar[hook,shift left=1cm]{d} &
  \C^{\widehat I} \ar[hook]{d}\\
  \Z_P\times T^r \cong 
  \Z_{\mathrm A,b} \ar[hook]{r} & \C^m
\end{tikzcd}
\]
has trivial normal bundle. Since every redundant inequality in the system $(\mathrm A,b)$ is also redundant in the system $(\mathrm A_{\widehat I},b_{\widehat I})$ we also obtain an embedding $\Z_G\hookrightarrow\Z_G\times T^{r_G-r}\hookrightarrow\Z_P$ with trivial normal bundle.
\end{proof}

An affine plane $L$ of codimension $k$ in $W$ is called \emph{transverse} to the polytope $P=P(\mathrm A,b)$ given by a generic system of inequalities~\eqref{ptope} if $L$ does not intersect $(k-1)$-dimensional faces of the polytope~$P$. In particular, a hyperplane is transverse if it does not intersect the vertices of the (simple) polytope~$P$. The intersection $P\cap L$ with a transverse plane is a simple polytope in~$L$.

\begin{prop}\label{qtriv}
Let $L$ be a transverse plane to the polytope $P=P(\mathrm A,b)$ given by a generic system of inequalities~\eqref{ptope} and $Q=L\cap P$ is nonempty. Then $\Z_Q$ embeds as a smooth submanifold in $\mathcal Z_P$ with trivial normal bundle.
\end{prop}

\begin{proof}
We proceed by induction on the codimension $k$ of the plane $L$.

Let $k=1$. Then the hyperplane $L$ is given by an equation $\langle a_{m+1},w\rangle+b_{m+1}=0$ with nonzero $a_{m+1}\in W^*$. Let $P'$ be the polytope in $W$ given by the system of inequalities $\langle a_i,w\rangle+b_i\ge0$, $i=1,\ldots,m+1$. Since $L$ is transverse, this is a generic system. Hence
$P'$ is a simple polytope and $Q$ is one of its facets. Shifting the origin to a point of $Q$ we may assume that $b_{m+1}=0$. By Proposition~\ref{ftriv}, $\Z_Q$ embeds into $\Z_{P'}$ with trivial normal bundle. Let us look at the intersection of quadrics in $\C^{m+1}$ realizing this embedding. A basis of the space $V'$ of linear relations on the vectors $a_1,\ldots,a_{m+1}$ is obtained from a basis of the space $V$ of linear relations on the vectors $a_1,\ldots,a_m$ by adding one relation $r_1a_1+\cdots+r_ma_m+a_{m+1}=0$. Therefore the intersection of quadrics~\eqref{intqu} corresponding to the polytope~$P'$ is
\[
  \bigl\{ z \in \C^{m+1} \colon
  \gamma_1|z_1|^2+\dots+\gamma_m|z_m|^2=\delta,\quad
  r_1|z_1|^2+\dots+r_m|z_m|^2+|z_{m+1}|^2=\delta_{m+1} \bigr\},
\]
where $\{\gamma_1,\ldots,\gamma_m\}$ is the Gale dual to 
$\{a_1,\ldots,a_m\}$, $\delta=b_1\gamma_1+\dots+b_m\gamma_m$ and $\delta_{m+1}=b_1r_1+\dots+b_mr_m$.
By Proposition~\ref{ftriv}, adding the condition $z_{m+1}=0$ to the above equations yields a nondegenerate system of quadrics defining the manifold $\Z_Q\times T^{r_Q}$ (where $r_Q$ is the number of hyperplanes $\langle a_i,w\rangle+b_i=0$ that do not intersect~$Q$). On the other hand, this same system of quadrics is obtained from the system $\gamma_1|z_1|^2+\dots+\gamma_m|z_m|^2=\delta$ defining $\Z_P$ in $\C^m$ by adding one equation $r_1|z_1|^2+\dots+r_m|z_m|^2=\delta_{m+1}$. Hence $\Z_Q\times T^{r_Q}$ is a codimension-one submanifold of $\Z_P$ with trivial normal bundle and therefore $\Z_Q$ also has trivial normal bundle in~$\Z_P$.

Now let $L$ be a transverse plane of codimension $k$. Take a hyperplane $H$ containing $L$ and transverse to $P$ (i.\,e. not intersecting the vertices of~$P$). Let $Q=L\cap P$ and $S=H\cap P$. By the previous argument, $\Z_S$ embeds into $\Z_P$ with trivial normal bundle. At the same time, $L$ is a transverse plane to $S$ of codimension $(k-1)$. Therefore, by the induction hypothesis, $\Z_Q$ embeds into $\Z_S$ with trivial normal bundle.
\end{proof}

For $J = \{j_1, \ldots, j_k\} \ss [m]$, we denote by $T^J$ the corresponding coordinate subtorus in $T^m$:
\[
  T^J=\{(t_1,\ldots,t_m)\in T^m\colon t_i=1\text{ for }i\notin J\}.
\]
Then, by Theorem~\ref{zpqua}, the intersection of quadrics~\eqref{intqu} defines an embedding $\Z_P\times T^{R}$ into $\C^m$, where for each $i\in R$ the inequality $\langle a_i,w\rangle+b_i\ge0$ in~\eqref{ptope} is redundant. Furthermore, Proposition~\ref{ftriv} gives an embedding $\Z_G\times T^{R_G\setminus R}$ into $\Z_P$, where the elements of $R_G\setminus R$ correspond to facets of $P$ that do not intersect the face~$G$. Finally, Proposition~\ref{qtriv} gives an embedding $\Z_Q\times T^J$ into $\Z_P$ where the elements of $J$ correspond to facets of $P$ that do not intersect~$Q$.

For a simple polytope $P$, the dual simplicial complex $\K_P$ is defined as the nerve of the cover of the boundary of~$P$ by its facets. Thus, if $P$ is given by the system of inequalities~\eqref{ptope}, then
\begin{equation}\label{nerve}
\begin{aligned}
  \K_P&=\{I\subset[m]\colon \bigcap_{i\in I}F_i\ne\varnothing\}\\
  &=\{I\subset[m]\colon \text{ there exists }w\in P
  \text{ such that }\langle a_i,w\rangle+b_i=0
  \text{ for }i\in I\}.
\end{aligned}
\end{equation}
This definition extends to an arbitrary polyhedron $P=P(\mathrm A,b)$ given by a system of inequalities~\eqref{ptope}. Moreover, if the system $(\mathrm A,b)$ is generic then the redundant inequalities $\langle a_i,w\rangle+b_i\ge0$ correspond to ghost vertices $\{i\}\notin\K_P$.

A \emph{missing face} (or \emph{minimal nonface}) of a simplicial complex $\K$ is a subset $I\subset[m]$ such that $I$ is not a simplex of $\K$, but every proper subset of $I$ is a simplex of~$\K$. A missing face corresponds to the full subcomplex $\K_I=\partial\Delta_I\subset\K$, where $\partial\Delta_I$ is the boundary of the simplex $\Delta_I$ on the vertex set~$I$. The simplicial cycle $\partial\Delta_I$ corresponding to the missing face $I$ represents a generator of the group $\mathcal H_{|I|-2,I}=\wt H_{|I|-2}(\partial\Delta_I)\cong\mathbb Z$. By the isomorphism from Theorem~\ref{Hochster} the simplicial cycle $\partial\Delta_I$ corresponds to a nonzero class (a basis element) in $H_{2|I|-1}(\Z_\K)$.
Moreover, for every $J\subset[m]\setminus I$ the simplicial cycle $\partial\Delta_I$ represents a class (possibly zero) in the group $\mathcal H_{|I|-2,I\sqcup J}=\wt H_{|I|-2}(\K_{I\sqcup J})$ and thereby represents a class in the group $H_{2|I|+|J|-1}(\Z_\K)$ by the isomorphism from Theorem~\ref{Hochster}. We call the classes of $H_*(\Z_\K)$ obtained in this way \emph{corresponding to missing faces} of~$\K$.

The isomorphism from Theorem~\ref{Hochster} is induced by an isomorphism between the groups of simplicial chains of all full subcomplexes of $\K$ and the groups of cellular chains of the moment-angle complex $\Z_\K=(D^2,S^1)^\K$. Under this isomorphism, the simplicial cycle $\partial\Delta_I$ of $\K_{I \sqcup J}$ maps to the fundamental class of the subcomplex
\[
  \bigcup_{j\in I}\Bigl(\prod_{i\in I\setminus j}D^2\times \prod_{i\in J\sqcup j}S^1\Bigr)=\Z_{\partial\Delta_I}\times T^J\cong S^{2k-1}\times T^J.
\]
The following lemma shows that in the polytopal case the corresponding homology classes of $\Z_P$ are represented by embedded smooth submanifolds $S^{2k-1}\times T^J$ with trivial normal bundle.

\begin{lemma}\label{normal-mf-submanifold}
Let $I$ be a missing face of $\K_P$ with $|I|=k$.
Then there exist a smooth embedding $S^{2k-1}\x T^{[m]\sm I}\hra \Z_P$ with trivial normal bundle, such that its restriction to $S^{2k-1}\times T^J$ for any $J\subset[m]\sm I$ is a smooth submanifold representing the homology class in $H_{2k-1+|J|}(\Z_\mathcal K)$ corresponding to the missing face~$I$.
\end{lemma}

\begin{proof}
Without loss of generality we assume $I = \{1, 2, \ldots, k\}$, then $F_1, F_2, \ldots, F_k$ are the corresponding facets of $P = P(\mathrm A,b)$. 
For each $j\in I$ consider face $F_{I\setminus j} = \bigcap_{i \in I \sm j} F_i$ of~$P$, dual to the simplex $I \sm j$ of $\6\D_I\subset\K_P$. Choose a point $w_j$ in the relative interior of $F_{I\setminus j}$ for $j=1,\ldots,k$. Consider the $(k-1)$-dimensional plane $L$ containing the $k$ points $w_1, \ldots, w_k$. We claim that $Q=L\cap P$ is a $(k-1)$-dimensional simplex with vertices $w_1, \ldots, w_k$, so these points are affinely independent and $L$ is transverse to~$P$. Indeed, since $Q=L\cap P$ is convex we have $\mathop\mathrm{conv}(w_1, \ldots, w_k)\subset Q$. Suppose that $Q$ also contains some point $w'=\sum_{i=1}^k \lambda_iw_i$ where $\sum_{i=1}^k\lambda_i=1$ and $\lambda_j<0$ for some~$j$. Then
$\langle a_j, w' \rangle + b_j = \lambda_j( \langle a_j, w_j \rangle + b_j) < 0$ and therefore $w' \notin P$. A contradiction.

Now from Proposition~\ref{qtriv} we obtain an embedding $\Z_Q\times T^{[m]\sm I}$ in $\Z_P$ with trivial normal bundle, where $\Z_Q\cong S^{2k-1}$ since $Q$ is a $(k-1)$-dimensional simplex.

Through the identification of $\Z_P=(D^2,S^1)^{\mathcal K_P}$ with the intersection of quadrics \eqref{intqu}  (see~\cite[Theorem~6.2.4]{BP}) we observe that the submanifold $\mu^{-1}(i_{\mathrm A,b}(Q))\cong \Z_Q\times T^{[m]\sm I}$ corresponds to the subcomplex
\[
  \bigcup_{j\in I}\Bigl(\prod_{i\in I\setminus j}D^2\times \prod_{i\notin I\setminus j}S^1\Bigr)=\Z_{\partial\Delta_I}\times T^{[m]\setminus I}\cong S^{2k-1}\times T^{[m]\setminus I},
\]
which represents the homology class in $H_{k-1+m}(\mathcal Z_{\K_P})$ corresponding to the missing face~$I$. The same is true for any $J\subset[m]\setminus I$.
\end{proof}

\section{Three dimensional spheres: diffeomorphism}

Previously, the authors obtained the following result.

\begin{thm}[{\cite[Theorem 4.4]{kov-pa}}]\label{chordal4iff}
Let $\K$ be a 3-dimensional simplicial sphere. There is a ring isomorphism $H^*(\Z_\K) \cong H^*(M_1\#\cdots\# M_k)$ where each $M_i$ is a product of spheres, if and only if one of the following conditions is satisfied:
\begin{itemize}
\item[(a)] $\K=S^0*S^0*S^0*S^0$ (the boundary of a $4$-dimensional cross-polytope);
\item[(b)] $\K^1$ is a chordal graph;
\item[(c)] $\K^1$ has exactly two missing edges and they are not adjacent to each other (i.\,e. form a chordless $4$-cycle).
\end{itemize}
\end{thm}

It is natural to ask whether the criterion of Theorem~\ref{chordal4iff} can be strengthened to a diffeomorphism in the case where the moment-angle manifold $\mathcal{Z}_\K$ is smooth.
The following theorem shows that if $\K$ is a \textit{polytopal} sphere then in case~(c) the diffeomorphism holds. Further in Theorem~\ref{diffeo4-starshaped} we show that the diffeomorphism also holds for starshaped spheres. As for case~(b), in Theorem~\ref{connected-starshaped} we establish the diffeomorphism under some additional assumption.

\begin{thm}\label{diffeo4}
Let $P$ be a $4$-dimensional simple polytope such that $3$-dimensional simplicial sphere $\K_P$ has exactly two missing edges and they are not adjacent to each other. Then there is a diffeomorphism $\Z_P \cong M_1\#\cdots\# M_k$ where each $M_i$ is a product of spheres and one of $M_i$ is a product of three spheres.
\end{thm}

In terms of polytope $P$ the condition means that $P$ has exactly two pairs of nonadjacent facets, which form a prismatic $4$-circuit. The first example of such polytope was given in~\cite{FCMW}.

To identify a manifold with a connected sum of products of spheres we use the following result from~\cite{gi-lo13}, which, in turn, relies on the $h$-cobordism theorem.

\begin{thm}[{\cite[Theorem A1.1]{gi-lo13}, \cite[Theorem 2.1]{iriy18}}]\label{h-cob}
Let $\mathcal Q$ be a compact smooth manifold of dimension $d+1 \geq 6$ with boundary $\6\mathcal Q$, satisfying the following:
\begin{enumerate}
\item $\mathcal Q$ and $\6\mathcal Q$ are simply connected and $H_i(\mathcal Q) = 0$ for $i \geq d - 1$.
  
\item There is a finite collection $\{X_j\}$ of disjointly embedded closed smooth submanifolds inside $\mathcal Q$ with trivial normal bundles, such that the embedding $\bigsqcup_j X_j \to\mathcal Q$ induces isomorphisms of integral homology groups of positive dimensions.
\end{enumerate}
Then $\mathcal Q$ is diffeomorphic to the boundary connected sum
$\coprod_j (X_j \x D^{d+1 - \dim{X_j}})$,
and therefore $\6\mathcal Q$ is diffeomorphic to the connected sum $\#_j (X_j \x S^{d - \dim{X_j}})$.

\end{thm}

\begin{proof}[Proof of Theorem~\ref{diffeo4}]
Let the polytope $P$ be given by a system~\eqref{ptope} without redundant inequalities, so that $P$ has $m$ facets. We identify $\Z_P$ with the nondegenerate intersection of quadrics
\[
  \bigl\{ z \in \C^m \colon
  \gamma_1|z_1|^2+\dots+\gamma_m|z_m|^2=\delta\bigr\},
\]
see~\eqref{intqu}. The proof is based on the application of the Theorem~\ref{h-cob}, for this we need to provide a manifold $\mathcal Q$ satisfying its conditions, with boundary $\6 \mc{Q} = \Z_P$ and a collection of submanifolds $\{X_j\}$.

Following~\cite[\S 1]{gi-lo13}, define
\begin{equation}\label{define-Q}
  \mathcal Q=\bigl\{ (x,z) \in \R_\ge\times\C^m \colon
  \gamma_1(x^2+|z_1|^2)+|z_2|^2+\dots+\gamma_m|z_m|^2
  =\delta\bigr\}.
\end{equation}
Then $\mathcal Q$ is a smooth manifold with boundary $\6\mathcal Q =\mathcal  Q \cap \{x = 0\} = \Z_P$. Also note that $\mathcal Q \cap \{x = c\} \cong \Z_P$ for small $c\ge0$. Denote by $i\colon \Z_P \hra\mathcal  Q$ the embedding of the boundary and by $i_c$ the embedding $\Z_P \cong\mathcal  Q \cap \{x = c\} \hra\mathcal  Q$. The normal bundles of these embeddings are obviously trivial.
    
Let $\K=\K_P$. Similarly to homeomorphism $\mathcal Z_P\cong (D^2,S^1)^\K$ we have a homeomorphism $\mathcal Q \cong (\bm{X}, \bm{A})^\K$, where $(\bm{X}, \bm{A})^\K$ is a polyhedral product (\cite[Construction~4.2.1]{BP}) corresponding to the collection of pairs $(X_1, A_1) = (D^3_+, S^2_+)$ and $(X_j, A_j) = (D^2, S^1)$ for $j >1$. Here $D^3_+=\{(x,z)\in\R_\ge\times\C\colon x^2+|z|^2\le1\}$ and $S^2_+=\{(x,z)\in\R_\ge\times\C\colon x^2+|z|^2=1\}$. Since the pair $(D^3_+, S^2_+)$ is contractible, there is a homotopy equivalence $\mathcal Q \simeq (D^2,S^1)^{\K_{[m] \sm 1}}=\Z_{\K_{[m] \sm 1}}$. We also have a retraction $\Z_\K\to\Z_{\K_{[m] \sm 1}}$, where $\Z_\K=\partial\mathcal Q$ and $\Z_{\K_{[m] \sm 1}}\simeq\mathcal Q$. Therefore, the embedding of the boundary $i\colon\partial\mathcal Q\hra\mathcal Q$ is surjective on the homology groups. 
    
Condition (1) of Theorem~\ref{h-cob} is satisfied since the moment-angle complexes $\Z_\K=\partial\mathcal Q$ and $\Z_{\K_{[m] \sm 1}}\simeq\mathcal Q$ are two-connected (\cite[Proposition~4.3.5]{BP}), and Poincar\'e duality holds for $\Z_\K$.
        
To satisfy condition (2) of Theorem~\ref{h-cob} it is sufficient to provide a set of submanifolds $\{X_j\}$ in $\Z_P$ with trivial normal bundles, such that the embedding $\bigsqcup_j X_j \to \Z_P \overset{i}{\hra}\mathcal Q$ induces isomorphisms of integral homology groups of positive dimensions. (The submanifolds $X_j$ can be assumed to be parwise disjoint in $\mathcal Q$ by considering the embeddings $ X_j \to \Z_P \overset{i_{c_j}}{\hra}\mathcal Q$ with pairwise different small constants~$c_j$.)
        
Now we analyse the homology groups $H_*(\mathcal Q)$. We have $\H_{l,*}(\K_{[m] \sm 1})=0$ for $l\ge 3$, since $\K$ is a 3-dimensional sphere. That is,
\[
  \widetilde H_*(\mathcal Q) \cong \H_{0,*}(\K_{[m] \sm 1})\oplus\H_{1,*}(\K_{[m] \sm 1})\oplus\H_{2,*}(\K_{[m] \sm 1}).
\]  
  
By assumption $\K$ has exactly two missing edges which form a single chordless $4$-cycle $C$. Without loss of generality we assume that $\{1\} \notin C$. Then $\H_{2,*}(\K_{[m] \sm 1}) = 0$. 
Indeed, if there exist $I \ss [m]\setminus 1$ such that $\H_{2,I} \neq 0$ then we have $\H^{0,[m] \sm I} \cong \H_{2,I} \neq 0$ due to Poincar\'e duality~\eqref{aldua}, but $1 \in [m] \sm I$ and hence $\K_{[m] \sm I}$ is connected and $\H^{0,[m] \sm I} = 0$, a contradiction. 
    
The group $\wt{H}_0(\K_I)$ is nonzero is and only if $I$ is one of the two missing edges, since otherwise the full subcomplex $\K_I$ is connected. Therefore $\H_{0,*}(\K_{[m] \sm 1}) \cong \H_{0,*}(C) \cong \mb{Z}\langle a_1, a_2 \rangle$, where $a_1$ and $a_2$ correspond to the missing edges of~$C$.

From Poincar\'e duality we obtain $\H_{1,I}=\wt{H}_1(\K_I) \cong \wt{H}^1(\K_{[m] \sm I})$ for any $I\subset[m]$, hence the group $\H_{1,*}(\K_{[m] \sm 1})$ is torsion-free. Also $\wt{H}_1(C) \cong \mb{Z}\langle c \rangle$, and for $\K_I \neq C$ every nonzero cycle $\gamma \in \wt{H}_1(\K_I)$ can be written as $\gamma = \lambda_1 \gamma_1 + \cdots + \lambda_k\gamma_k$, where each $\gamma_i$ corresponds to a missing 2-face of $\K_I$ and $\lambda_i\in\mathbb Z$.
Hence, $\H_{1,*}(\K_{[m] \sm 1}) \cong \mb{Z}\langle c, b_\alpha \colon \alpha \in A\rangle$, where the classes $\{b_\alpha\colon\alpha \in A\}$ correspond to missing 2-faces. Finally we obtain
\[
  \widetilde H_*(\mathcal Q) = \mb{Z}\langle a_1, a_2, c \rangle \oplus \mb{Z}\langle b_\alpha \colon \alpha \in A \rangle \cong 
  \widetilde H_*\bigl( S^3 \x S^3 \bigr) \oplus 
  \widetilde H_*\Bigl( \bigvee_{\alpha \in A} S^{n_\alpha} \Bigr).
\]

Set one of the submanifolds $X_j$ equal to $S^3 \x S^3 = \Z_C$, it corresponds to the full subcomplex $C$. The embedding $\Z_C \hra \mc{Q}$ has trivial normal bundle by Lemma~\ref{by-Iriye} below and induces a natural isomorphism between $\widetilde H_*(\Z_C)$ and the subgroup of $H_*(\mathcal Q)$ generated by the elements $a_1, a_2$ and $c$.

The remaining generators $b_\alpha$ of the group $\widetilde H_*(\mathcal Q)$ correspond to missing 2-faces of $\K_P$. They are represented by embedded spheres with trivial normal bundles by Lemma~\ref{normal_bundles} below.
\end{proof} 

\begin{lemma}[see \cite{iriy18}]\label{by-Iriye}
Let $M$ be a smooth manifold of dimension $d$, where $d \geq 13$, and let $X$ be an embedded submanifold in $M$ diffeomorphic to $S^3 \x S^3$. Then the normal bundle $\nu(X \hra M)$ is trivial.  
\end{lemma}

\begin{proof}
The normal bundle of $X$ in $M$ is classified by the set of homotopy classes $[S^3 \x S^3, BO(d-6)]$, which is trivial (consists of one element).
To see this, consider the exact sequence associated with the cofibration
$S^3 \vee S^3 \to S^3 \x S^3 \to S^6$:
\[
  [S^6, BO(d-6)] \to [S^3 \x S^3, BO(d-6)] \to [S^3 \vee S^3, BO(d-6)].
\]
We have $[S^6,BO(d-6)]\cong[S^5,O(d-6)]= 0$ for $d \geq 13$ since $\pi_5(O)=0$. 
Moreover,
\begin{multline*}
  [S^3 \vee S^3, BO(d-6)] \cong [S^3, BO(d-6)] \x [S^3, BO(d-6)] \cong\\
  \cong [S^2,O(d-6)] \x [S^2, O(d-6)] = 0   \q \text{for } d \geq 10,
\end{multline*}
since $\pi_2(O)=0$. Therefore $[S^3 \x S^3, BO(d-6)] = *$ for $d \geq 13$.
\end{proof}

\begin{rem}
In the context of Theorems~\ref{diffeo4} and~\ref{diffeo4-starshaped} we consider $M = \mc{Q}$, where $\dim \mc{Q} = m+n+1 = m+5$. Then the condition $d \geq 13$ is equivalent to $m \geq 8$.
The only polytope $P$ with $m < 8$ facets satisfying the assumptions of Theorem~\ref{diffeo4} is $P = I^2 \x \D^2$, for which $\Z_P \cong S^3 \x S^3 \x S^5$. As for Theorem~\ref{diffeo4-starshaped}, the condition $m \geq 8$ is inessential.
\end{rem}

Lemma~\ref{normal-mf-submanifold} allows us to represent homology classes of $H_*(\Z_P)$ corresponding to missing faces by submanifolds $S^{2k-1}\times T^J$ with trivial normal bundles. To represent these homology classes by spheres we need the following construction.

The \emph{gyration} $\mc{G}_k(X)$ of a topological manifold $X$ of dimension $d$ is a topological manifold
\[
  \mc{G}_k(X) = \bigl((X \sm D^d) \x S^{k-1}\bigr) 
  \cup_{S^{d-1} \x S^{k-1}}(S^{d-1} \x D^k)
\]
of dimension $d+k-1$. We also consider iterated ($l$-fold) gyration
$\mc{G}^l_k(X) = \mc{G}_k(\mc{G}_k( \ldots \mc{G}_k(X)\ldots))$ ($l$ times).

We only consider $\mc{G}_2(X)$ below, and denote it simply by $\mc{G}(X)$. If $X$ is a smooth manifold, then $\mc{G}_k(X)$ can also be made smooth by smoothing the corners of $S^{d-1} \x S^{k-1}$, the details of this smoothing are provided below.

It follows from the definition that $\mc{G}(S^d) \cong S^{d+1}$.

\begin{lemma} \label{normal_bundles}
For a missing face $I$ of $\K_P$ and $J\subset[m]\setminus I$, where $|I|=k$ and $|J|=l$, consider the embedding $i\colon S^{2k-1} \x T^J \hra \Z_P$ from Lemma~\ref{normal-mf-submanifold}.
Then there exist a smooth embedding of an $l$-fold gyration $g\colon\mc{G}^{l}(S^{2k-1}) \hra \Z_P$ with trivial normal bundle such that embeddings $i$ and $g$ are bordant. Therefore, the homology class in $H_{2k-1+l}(\Z_P)$ corresponding to the missing face $I$ is represented by an embedded sphere $g(\mc{G}^{l}(S^{2k-1}))\cong S^{2k-1+l}$.
\end{lemma}

\begin{proof}
To simplify notation, we identify the polytope $P$ with its image $i_{\mathrm A,b}(P)$ in $\R^m_\ge$ and identify $\Z_P$ with the intersection of quadrics~\eqref{intqu}. Then we have $\Z_P=\mu^{-1}(P)$ and $i(S^{2k-1} \x T^{[m]\setminus I})=\mu^{-1}(Q)$, where $Q\cong\Delta^{k-1}$ is a simplex (see Lemma~\ref{normal-mf-submanifold}). Denote $X:=i(S^{2k-1})=\Z_Q$ (the embedded sphere), then $i(S^{2k-1} \x T^{[m]\setminus I})=T^m\cdot X=T^{[m]\setminus I}\cdot X$ is the union of orbits of points of $X$ under the action of the coordinate subtorus $T^{[m]\setminus I}$ (note that $X$ is invariant under the action of~$T^I$) and similarly $i(S^{2k-1} \x T^J)=T^J\cdot X$.

\smallskip

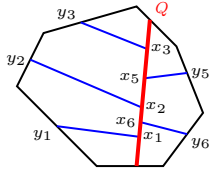
\begin{minipage}{0.3\textwidth}
\tikzset{every picture/.style={line width=0.75pt, x=1pt,y=1pt, baseline=(current bounding box.center)}}
\begin{tikzpicture}

\coordinate (A1) at (30,0);    
\coordinate (A2) at (0,30);    
\coordinate (A3) at (10,50);    
\coordinate (A4) at (45,60);    
\coordinate (A5) at (60,45);    
\coordinate (A6) at (70,20);    
\coordinate (A7) at (55,0);    

\coordinate (D1) at (50,55);    
\coordinate (D2) at (45,0);    

\coordinate (X1) at (15,15);
\coordinate (X2) at (5,40);
\coordinate (X3) at (24,54);
\coordinate (X4) at (64,35);
\coordinate (X5) at (64,12);

\coordinate (Y3) at (49,44);
\coordinate (Y4) at (48,33);
\coordinate (Y2) at (47,22);
\coordinate (Y5) at (46.5,16.5);
\coordinate (Y1) at (46,11);

\draw (A1) -- (A2) -- (A3) -- (A4) -- (A5) -- (A6) -- (A7) -- cycle;

\draw[color=blue] (X1) -- (Y1);
\draw[color=blue] (X2) -- (Y2);
\draw[color=blue] (X3) -- (Y3);
\draw[color=blue] (X4) -- (Y4);
\draw[color=blue] (X5) -- (Y5);

\draw[color=red, line width=1.5pt] (D1) -- (D2);

\begin{scope}[every node/.style={font=\tiny, inner sep=1pt}]
\node[below left] at (X1) {$y_1$};
\node[left] at (X2) {$y_2$};
\node[above left] at (X3) {$y_3$};
\node[right] at (X4) {$y_5$};
\node[below right] at (X5) {$y_6$};

\node[right] at (Y1) {$x_1$};
\node[right] at (Y2) {$x_2$};
\node[right] at (Y3) {$x_3$};
\node[left] at (Y4) {$x_5$};
\node[left] at (Y5) {$x_6$};
\end{scope}

\begin{scope}[every node/.style={color=red, font=\tiny, inner sep=1pt}]
\node[above right] at (D1) {\textit Q 
};
\end{scope}

\end{tikzpicture}
\end{minipage} 
\quad
\begin{minipage}{0.63\textwidth}
Choose points $x_j \in \oname{relint} Q$ and $y_j \in \oname{relint}F_j$, $j \in J$, such that the segments $[x_j, y_j]$ do not intersect (it is always possible since $P$ is convex and $\dim Q < \dim P$, see the figure on the left). Then 
\[
  Y_j := \mu^{-1}([x_j, y_j]) \cong D^2 \x T^{m-1}
\]  
is a smooth submanifold with boundary in $\Z_P$.
Then $\mu(T^{[m]\sm I}\cdot X) \cap \mu(Y_j) = x_j$, which implies
\[
  (T^{[m]\sm I}\cdot X)\cap Y_j=\mu^{-1}(x_j)\cong T^m.
\]
\end{minipage}

\smallskip

For the given $J\subset[m]\setminus I$, consider also the submanifold
$Y_{j,J} \cong D^2_j \x T^{J \sm j}$ in $Y_j\cong D^2_j\times T^{[m]\setminus j}$ and denote by $f_j \colon D^2_j \x T^{J\sm j} \hra \Z_P$ the corresponding embedding, $f_j(D^2_j \x T^{J\sm j})=Y_{j,J}$. The submanifolds $X$ and $Y_{j,J}$ intersect at the point $p_j = X\cap Y_{j,J}$, which belongs to the boundary $\partial Y_{j,J}\cong T^J$.

Since the normal bundle of the embedding $f_j \colon D^2_j \x T^{J\sm j} \hra \Z_P$ is trivial, for a tubular neighbourhood $V$ of submanifold $Y_{j,J}$ we have
\[
  V\cong Y_{j,J} \x D^{m+n-l-1}\cong D^{m+n-l+1} \x T^{J \sm j}.
\]  
By the compactness of $T^J$ the intersection $V \cap X$ contains some ball neighbourhood $D^{2k-1}\cong U \ss X$ of the point $p_j$ such that $V \cap (T^J\cdot X)$ contains $T^J\cdot U$ (see Fig.~\ref{fig:XUVDS}).

\begin{figure}[htbp]
    \centering
\tikzset{every picture/.style={line width=0.75pt}} 
\begin{tikzpicture}[x=1.5pt,y=1.5pt,yscale=-0.35,xscale=0.35]

\draw  [draw opacity=0][fill={rgb, 255:red, 245; green, 166; blue, 35 }  ,fill opacity=0.25 ] (137.15,158.69) -- (202.84,163.04) .. controls (211.37,163.61) and (216.64,177.75) .. (214.61,194.62) .. controls (212.57,211.49) and (204,224.71) .. (195.46,224.14) -- (129.77,219.79) .. controls (121.23,219.23) and (115.96,205.09) .. (118,188.22) .. controls (120.04,171.35) and (128.61,158.13) .. (137.15,158.69) -- cycle ;
\draw  [draw opacity=0][fill={rgb, 255:red, 245; green, 166; blue, 35 }  ,fill opacity=0.25 ] (527.21,113.22) -- (585.5,113.22) .. controls (593.07,113.22) and (599.21,126.87) .. (599.21,143.72) .. controls (599.21,160.56) and (593.07,174.22) .. (585.5,174.22) -- (527.21,174.22) .. controls (519.64,174.22) and (513.5,160.56) .. (513.5,143.72) .. controls (513.5,126.87) and (519.64,113.22) .. (527.21,113.22) -- cycle ;
\draw  [color={rgb, 255:red, 128; green, 128; blue, 128 }  ,draw opacity=1 ][dash pattern={on 3.75pt off 3pt on 7.5pt off 1.5pt}][line width=0.75]  (513.5,143.22) .. controls (513.5,126.65) and (520.89,113.22) .. (530,113.22) .. controls (539.11,113.22) and (546.5,126.65) .. (546.5,143.22) .. controls (546.5,159.79) and (539.11,173.22) .. (530,173.22) .. controls (520.89,173.22) and (513.5,159.79) .. (513.5,143.22) -- cycle ;
\draw  [color={rgb, 255:red, 65; green, 117; blue, 5 }  ,draw opacity=1 ][fill={rgb, 255:red, 65; green, 117; blue, 5 }  ,fill opacity=0.2 ][dash pattern={on 5.63pt off 4.5pt}][line width=1.5]  (535.36,143.72) .. controls (535.36,127.15) and (544.76,113.72) .. (556.36,113.72) .. controls (567.95,113.72) and (577.36,127.15) .. (577.36,143.72) .. controls (577.36,160.29) and (567.95,173.72) .. (556.36,173.72) .. controls (544.76,173.72) and (535.36,160.29) .. (535.36,143.72) -- cycle ;
\draw  [color={rgb, 255:red, 128; green, 128; blue, 128 }  ,draw opacity=1 ][dash pattern={on 3.75pt off 3pt on 7.5pt off 1.5pt}][line width=0.75]  (564.71,143.22) .. controls (564.71,126.65) and (572.43,113.22) .. (581.96,113.22) .. controls (591.49,113.22) and (599.21,126.65) .. (599.21,143.22) .. controls (599.21,159.79) and (591.49,173.22) .. (581.96,173.22) .. controls (572.43,173.22) and (564.71,159.79) .. (564.71,143.22) -- cycle ;
\draw  [dash pattern={on 4.5pt off 4.5pt}] (443,130.72) .. controls (443,114.15) and (452.18,100.72) .. (463.5,100.72) .. controls (474.82,100.72) and (484,114.15) .. (484,130.72) .. controls (484,147.29) and (474.82,160.72) .. (463.5,160.72) .. controls (452.18,160.72) and (443,147.29) .. (443,130.72) -- cycle ;
\draw  [color={rgb, 255:red, 128; green, 128; blue, 128 }  ,draw opacity=1 ][dash pattern={on 4.5pt off 4.5pt}] (118,188.22) .. controls (119.66,171.37) and (131.3,157.72) .. (144,157.72) .. controls (156.7,157.72) and (165.66,171.37) .. (164,188.22) .. controls (162.34,205.06) and (150.7,218.72) .. (138,218.72) .. controls (125.3,218.72) and (116.34,205.06) .. (118,188.22) -- cycle ;
\draw   (17,138.86) .. controls (17,91.99) and (94.45,54) .. (190,54) .. controls (285.55,54) and (363,91.99) .. (363,138.86) .. controls (363,185.73) and (285.55,223.72) .. (190,223.72) .. controls (94.45,223.72) and (17,185.73) .. (17,138.86) -- cycle ;
\draw   (93,131.22) .. controls (93,114.37) and (134.86,100.72) .. (186.5,100.72) .. controls (238.14,100.72) and (280,114.37) .. (280,131.22) .. controls (280,148.06) and (238.14,161.72) .. (186.5,161.72) .. controls (134.86,161.72) and (93,148.06) .. (93,131.22) -- cycle ;
\draw  [color={rgb, 255:red, 208; green, 2; blue, 27 }  ,draw opacity=1 ][line width=1.5]  (47,131.22) .. controls (47,100.84) and (109.46,76.22) .. (186.5,76.22) .. controls (263.54,76.22) and (326,100.84) .. (326,131.22) .. controls (326,161.59) and (263.54,186.22) .. (186.5,186.22) .. controls (109.46,186.22) and (47,161.59) .. (47,131.22) -- cycle ;
\draw  [color={rgb, 255:red, 65; green, 117; blue, 5 }  ,draw opacity=1 ][fill={rgb, 255:red, 65; green, 117; blue, 5 }  ,fill opacity=0.3 ][dash pattern={on 5.63pt off 4.5pt}][line width=1.5]  (145,192.22) .. controls (145,175.37) and (155.3,161.72) .. (168,161.72) .. controls (180.7,161.72) and (191,175.37) .. (191,192.22) .. controls (191,209.06) and (180.7,222.72) .. (168,222.72) .. controls (155.3,222.72) and (145,209.06) .. (145,192.22) -- cycle ;
\draw  [draw opacity=0][line width=1.5]  (175.35,221.13) .. controls (173.04,222.16) and (170.57,222.72) .. (168,222.72) .. controls (155.3,222.72) and (145,209.06) .. (145,192.22) .. controls (145,175.37) and (155.3,161.72) .. (168,161.72) .. controls (169.9,161.72) and (171.75,162.03) .. (173.52,162.6) -- (168,192.22) -- cycle ; \draw  [color={rgb, 255:red, 65; green, 117; blue, 5 }  ,draw opacity=1 ][line width=1.5]  (175.35,221.13) .. controls (173.04,222.16) and (170.57,222.72) .. (168,222.72) .. controls (155.3,222.72) and (145,209.06) .. (145,192.22) .. controls (145,175.37) and (155.3,161.72) .. (168,161.72) .. controls (169.9,161.72) and (171.75,162.03) .. (173.52,162.6) ;  
\draw [color={rgb, 255:red, 245; green, 166; blue, 35 }  ,draw opacity=1 ][line width=3]    (119.5,181) .. controls (134.5,184.22) and (151,187.72) .. (172.5,188.22) ;
\draw [color={rgb, 255:red, 0; green, 0; blue, 0 }  ,draw opacity=1 ]   (145.5,183.72) ;
\draw [shift={(145.5,183.72)}, rotate = 0] [color={rgb, 255:red, 0; green, 0; blue, 0 }  ,draw opacity=1 ][fill={rgb, 255:red, 0; green, 0; blue, 0 }  ,fill opacity=1 ][line width=0.75]      (0, 0) circle [x radius= 3.35, y radius= 3.35]   ;
\draw  [dash pattern={on 4.5pt off 4.5pt}]  (596.5,197.5) .. controls (625,201.22) and (628.5,93.22) .. (609.5,93.5) ;
\draw  [dash pattern={on 4.5pt off 4.5pt}]  (501.5,189.22) .. controls (533.5,197.72) and (520,94.72) .. (514.5,85.22) ;
\draw [color={rgb, 255:red, 245; green, 166; blue, 35 }  ,draw opacity=1 ]   (459.5,185.22) -- (515.89,143.41) ;
\draw [shift={(517.5,142.22)}, rotate = 143.45] [color={rgb, 255:red, 245; green, 166; blue, 35 }  ,draw opacity=1 ][line width=0.75]    (10.93,-3.29) .. controls (6.95,-1.4) and (3.31,-0.3) .. (0,0) .. controls (3.31,0.3) and (6.95,1.4) .. (10.93,3.29)   ;
\draw    (501.5,189.22) .. controls (528.5,204.72) and (554,207.72) .. (596.5,197.5) ;
\draw    (596.5,197.5) .. controls (586.5,191.72) and (579.5,106.72) .. (609.5,93.5) ;
\draw    (514.5,85.22) .. controls (539.5,88.22) and (569.5,106.22) .. (609.5,93.5) ;
\draw   (645.5,128.22) .. controls (645.5,111.65) and (654.15,98.22) .. (664.81,98.22) .. controls (675.48,98.22) and (684.13,111.65) .. (684.13,128.22) .. controls (684.13,144.79) and (675.48,158.22) .. (664.81,158.22) .. controls (654.15,158.22) and (645.5,144.79) .. (645.5,128.22) -- cycle ;
\draw [color={rgb, 255:red, 0; green, 0; blue, 0 }  ,draw opacity=1 ][line width=0.75]    (469,101.72) .. controls (517,121.72) and (621,115.22) .. (663,98.22) ;
\draw [color={rgb, 255:red, 0; green, 0; blue, 0 }  ,draw opacity=1 ][line width=0.75]    (460,161.22) .. controls (531.5,182.22) and (610.5,176.72) .. (667.75,158.22) ;
\draw [color={rgb, 255:red, 208; green, 2; blue, 27 }  ,draw opacity=1 ][line width=2.25]    (445.5,115.22) .. controls (483,140.22) and (593.5,144.72) .. (645.5,124.72) ;
\draw    (501.5,189.22) .. controls (491.5,183.44) and (484.5,98.44) .. (514.5,85.22) ;
\draw [color={rgb, 255:red, 245; green, 166; blue, 35 }  ,draw opacity=1 ][line width=3]    (513.5,134.72) .. controls (531,137.72) and (543,137.22) .. (566,137.22) ;
\draw  [draw opacity=0][line width=1.5]  (561.43,172.84) .. controls (559.81,173.41) and (558.11,173.72) .. (556.36,173.72) .. controls (544.76,173.72) and (535.36,160.29) .. (535.36,143.72) .. controls (535.36,127.15) and (544.76,113.72) .. (556.36,113.72) .. controls (558.47,113.72) and (560.51,114.17) .. (562.44,115) -- (556.36,143.72) -- cycle ; \draw  [color={rgb, 255:red, 65; green, 117; blue, 5 }  ,draw opacity=1 ][line width=1.5]  (561.43,172.84) .. controls (559.81,173.41) and (558.11,173.72) .. (556.36,173.72) .. controls (544.76,173.72) and (535.36,160.29) .. (535.36,143.72) .. controls (535.36,127.15) and (544.76,113.72) .. (556.36,113.72) .. controls (558.47,113.72) and (560.51,114.17) .. (562.44,115) ;  
\draw [color={rgb, 255:red, 0; green, 0; blue, 0 }  ,draw opacity=1 ]   (536.5,137.22) ;
\draw [shift={(536.5,137.22)}, rotate = 0] [color={rgb, 255:red, 0; green, 0; blue, 0 }  ,draw opacity=1 ][fill={rgb, 255:red, 0; green, 0; blue, 0 }  ,fill opacity=1 ][line width=0.75]      (0, 0) circle [x radius= 3.35, y radius= 3.35]   ;
\draw [color={rgb, 255:red, 65; green, 117; blue, 5 }  ,draw opacity=1 ]   (533.5,69.72) -- (548.78,109.35) ;
\draw [shift={(549.5,111.22)}, rotate = 248.92] [color={rgb, 255:red, 65; green, 117; blue, 5 }  ,draw opacity=1 ][line width=0.75]    (10.93,-3.29) .. controls (6.95,-1.4) and (3.31,-0.3) .. (0,0) .. controls (3.31,0.3) and (6.95,1.4) .. (10.93,3.29)   ;
\draw [color={rgb, 255:red, 208; green, 2; blue, 27 }  ,draw opacity=1 ]   (657,78.72) -- (627.57,125.03) ;
\draw [shift={(626.5,126.72)}, rotate = 302.43] [color={rgb, 255:red, 208; green, 2; blue, 27 }  ,draw opacity=1 ][line width=0.75]    (10.93,-3.29) .. controls (6.95,-1.4) and (3.31,-0.3) .. (0,0) .. controls (3.31,0.3) and (6.95,1.4) .. (10.93,3.29)   ;
\draw [color={rgb, 255:red, 0; green, 0; blue, 0 }  ,draw opacity=1 ]   (588,220.22) -- (540.56,143.92) ;
\draw [shift={(539.5,142.22)}, rotate = 58.13] [color={rgb, 255:red, 0; green, 0; blue, 0 }  ,draw opacity=1 ][line width=0.75]    (10.93,-3.29) .. controls (6.95,-1.4) and (3.31,-0.3) .. (0,0) .. controls (3.31,0.3) and (6.95,1.4) .. (10.93,3.29)   ;
\draw [color={rgb, 255:red, 245; green, 166; blue, 35 }  ,draw opacity=1 ]   (600,67.72) -- (582.12,122.82) ;
\draw [shift={(581.5,124.72)}, rotate = 287.98] [color={rgb, 255:red, 245; green, 166; blue, 35 }  ,draw opacity=1 ][line width=0.75]    (10.93,-3.29) .. controls (6.95,-1.4) and (3.31,-0.3) .. (0,0) .. controls (3.31,0.3) and (6.95,1.4) .. (10.93,3.29)   ;
\draw  [draw opacity=0] (470.55,158.9) .. controls (468.35,160.08) and (465.98,160.72) .. (463.5,160.72) .. controls (452.18,160.72) and (443,147.29) .. (443,130.72) .. controls (443,114.15) and (452.18,100.72) .. (463.5,100.72) .. controls (466.67,100.72) and (469.68,101.77) .. (472.36,103.66) -- (463.5,130.72) -- cycle ; \draw   (470.55,158.9) .. controls (468.35,160.08) and (465.98,160.72) .. (463.5,160.72) .. controls (452.18,160.72) and (443,147.29) .. (443,130.72) .. controls (443,114.15) and (452.18,100.72) .. (463.5,100.72) .. controls (466.67,100.72) and (469.68,101.77) .. (472.36,103.66) ;  
\draw    (596.5,197.5) .. controls (607,199.22) and (613.5,185.22) .. (617,169.72) ;
\draw    (619,108.72) .. controls (616,99.22) and (616.5,94.22) .. (609.5,93.5) ;
\draw  [dash pattern={on 4.5pt off 4.5pt}]  (114.5,235.22) .. controls (146.5,243.72) and (133,140.72) .. (127.5,131.22) ;
\draw    (114.5,235.22) .. controls (141.5,250.72) and (167,253.72) .. (209.5,243.5) ;
\draw    (114.5,235.22) .. controls (104.5,229.44) and (97.5,144.44) .. (127.5,131.22) ;
\draw  [dash pattern={on 4.5pt off 4.5pt}]  (209.5,243.5) .. controls (238,247.22) and (241.5,139.22) .. (222.5,139.5) ;
\draw    (209.5,243.5) .. controls (199.5,237.72) and (192.5,152.72) .. (222.5,139.5) ;
\draw    (127.5,131.22) .. controls (152.5,134.22) and (182.5,152.22) .. (222.5,139.5) ;
\draw    (232,154.72) .. controls (229,145.22) and (229.5,140.22) .. (222.5,139.5) ;
\draw  [color={rgb, 255:red, 128; green, 128; blue, 128 }  ,draw opacity=1 ][dash pattern={on 4.5pt off 4.5pt}] (171.5,193.22) .. controls (173.16,176.37) and (184.46,162.72) .. (196.75,162.72) .. controls (209.04,162.72) and (217.66,176.37) .. (216,193.22) .. controls (214.34,210.06) and (203.04,223.72) .. (190.75,223.72) .. controls (178.46,223.72) and (169.84,210.06) .. (171.5,193.22) -- cycle ;
\draw [color={rgb, 255:red, 245; green, 166; blue, 35 }  ,draw opacity=1 ]   (48.5,225.22) -- (116.76,186.7) ;
\draw [shift={(118.5,185.72)}, rotate = 150.56] [color={rgb, 255:red, 245; green, 166; blue, 35 }  ,draw opacity=1 ][line width=0.75]    (10.93,-3.29) .. controls (6.95,-1.4) and (3.31,-0.3) .. (0,0) .. controls (3.31,0.3) and (6.95,1.4) .. (10.93,3.29)   ;
\draw [color={rgb, 255:red, 0; green, 0; blue, 0 }  ,draw opacity=1 ]   (77,133.72) -- (138.86,176.58) ;
\draw [shift={(140.5,177.72)}, rotate = 214.72] [color={rgb, 255:red, 0; green, 0; blue, 0 }  ,draw opacity=1 ][line width=0.75]    (10.93,-3.29) .. controls (6.95,-1.4) and (3.31,-0.3) .. (0,0) .. controls (3.31,0.3) and (6.95,1.4) .. (10.93,3.29)   ;
\draw [color={rgb, 255:red, 65; green, 117; blue, 5 }  ,draw opacity=1 ]   (180,128.72) -- (170.64,156.32) ;
\draw [shift={(170,158.22)}, rotate = 288.73] [color={rgb, 255:red, 65; green, 117; blue, 5 }  ,draw opacity=1 ][line width=0.75]    (10.93,-3.29) .. controls (6.95,-1.4) and (3.31,-0.3) .. (0,0) .. controls (3.31,0.3) and (6.95,1.4) .. (10.93,3.29)   ;

\draw (221.5,244.22) node [anchor=north west][inner sep=0.75pt]   [align=left] {$\displaystyle V$};
\draw (33,222.72) node [anchor=north west][inner sep=0.75pt]   [align=left] {$\displaystyle U$};
\draw (292.5,166.72) node [anchor=north west][inner sep=0.75pt]   [align=left] {$\displaystyle X$};
\draw (289.5,205.22) node [anchor=north west][inner sep=0.75pt]   [align=left] {$\displaystyle S^{1}_j \cdot X$};
\draw (499.5,199.22) node [anchor=north west][inner sep=0.75pt]   [align=left] {$\displaystyle V$};
\draw (442,181.72) node [anchor=north west][inner sep=0.75pt]   [align=left] {$\displaystyle U$};
\draw (655.5,57.72) node [anchor=north west][inner sep=0.75pt]   [align=left] {$\displaystyle X$};
\draw (635.5,164.72) node [anchor=north west][inner sep=0.75pt]   [align=left] {$\displaystyle S^{1}_j \cdot X$};
\draw (522.5,51.72) node [anchor=north west][inner sep=0.75pt]   [align=left] {$\displaystyle D^{2}$};
\draw (588,212.22) node [anchor=north west][inner sep=0.75pt]   [align=left] {$\displaystyle p_j$};
\draw (577,49.22) node [anchor=north west][inner sep=0.75pt]   [align=left] {$\displaystyle U\times D^{2}$};
\draw (66.5,116.22) node [anchor=north west][inner sep=0.75pt]   [align=left] {$\displaystyle p_j$};
\draw (175.5,109.22) node [anchor=north west][inner sep=0.75pt]   [align=left] {$\displaystyle D^{2}$};
\end{tikzpicture}

    \caption{An illustration of the embeddings of the objects considered in Lemma~\ref{normal_bundles}.}
    \label{fig:XUVDS}
\end{figure}
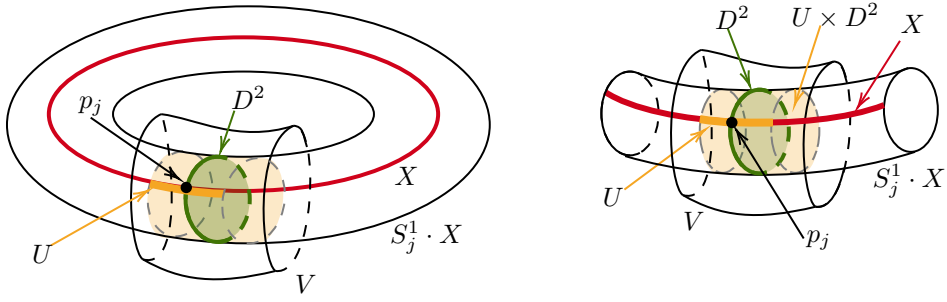

The embedding $f_j \colon D^2_j \x T^{J\sm j}\hra V \subset \Z_P$ and the embedding $T^J\cdot U\subset V\cap(T^J\cdot X)\subset\Z_P$ agree on the intersection $Y_{j,J}\cap (T^J\cdot U)=T^J\cdot p_j$. Therefore they can be extended to an embedding
\[
  f'_j\colon D^{2k-1} \x D^2_j \x T^{J \sm j} \hra V,
\]  
where $f'_j(D^{2k-1} \x D^2_j \x T^{J \sm j})\cong U\times Y_{j,J}$. 

Then we have an embedding
\begin{equation}\label{bordism}
  b\colon (S^{2k-1} \x S^1_j \cup_{D^{2k-1} \x S^1} D^{2k-1} \x D^2_j) 
  \x T^{J\sm j} \hra \Z_P
\end{equation}
with image $(T^J\cdot X)\cup_{T^J\cdot U} (U\x Y_{j,J})$,
where the restriction of $b$ to $S^{2k-1}\x S^1_j\x T^{J\sm j}= S^{2k-1}\x T^J$ coincides with $i\colon S^{2k-1} \x T^{J} \hra \Z_P$, and the restriction of $b$ to $D^{2k-1} \x D^2_j\x T^{J\sm j}$ coincides with $f'_j$. Restricting the embedding $b$ to the gyration $\mc{G}(S^{2k-1})$ we obtain an embedding
\begin{equation}\label{g'}
\begin{gathered}
  g'\colon \mc{G}(S^{2k-1})\x T^{J\sm j}=\bigl((S^{2k-1}\sm D^{2k-1})\x S^1\cup_{S^{2k-2}\x S^1}S^{2k-2}\x D^2\bigr) \x T^{J\sm j} \xrightarrow{\cong}\\ 
  \xrightarrow{\cong} \bigl(T^J\cdot (X\sm U)\bigr)\cup_{T^J\cdot \partial U} 
  \bigl(\6 U\x Y_{j,J}\bigr)
  \ss (T^J\cdot X)\cup_{T^J\cdot U} (U\x Y_{j,J})\subset\Z_P.
\end{gathered}
\end{equation}
However, this embedding is not smooth along $ T^J\cdot \partial U$ and needs to be smoothened (see Fig.~\ref{fig:GX-angles}), and we also need to glue together the trivial normal bundles over the parts $T^J\cdot (X\sm U)$ and $\6 U\x Y_{j,J}$ in~$\Z_P$.

Since $T^J\cdot (X\sm U)\subset T^J\cdot X=i(S^{2k-1} \x T^J)$, the normal bundle to $T^J\cdot (X\sm U)$ in $\Z_P$ is the restriction of the trivial normal bundle $\nu(i\colon S^{2k-1} \x T^J \hra \Z_P)$. 
First consider the case $J=\varnothing$, i.\,e. the normal bundle of the sphere $X=\mathcal Z_Q\cong S^{2k-1}$ in $\mathcal Z_P$. This embedding of $S^{2k-1}$ into $\mathcal Z_P$ is a restriction of the embedding $i\colon S^{2k-1} \x T^{[m]\sm I} \hra \Z_P$ from Lemma~\ref{normal-mf-submanifold}, so we have a decomposition
\[
  \nu(i\colon S^{2k-1}\hra\Z_P)\cong\tau^{[m] \sm I}\oplus \nu(i\colon S^{2k-1} \x T^{[m]\sm I} 
  \hra \Z_P)|_X=:\tau^{[m] \sm I}\oplus\nu',
\]
where $\tau^{[m] \sm I}$ is the trivial tangent bundle along the orbits of the action of the torus~$T^{[m] \sm I}$ and $\nu'$ is the trivial bundle of dimension $m+n-(2k-1)-(m-k)=n-k+1$, which is identified with the normal bundle of $Q\cong\Delta^{k-1}$ in~$P$. Now for arbitrary $J\subset[m]\sm I$ we have
\[
  \nu(i\colon S^{2k-1} \x T^J \hra \Z_P) = \tau^{[m] \sm (I \cup J)} \oplus \nu',
\]
Choose a decomposition $\nu'=\eta^1\oplus\nu''$ into the sum of a one-dimensional trivial subbundle $\eta^1$, which at the point $p_j = X\cap Y_{j,J}\in\partial Y_{j,J}$ is generated by the normal to the boundary $\partial Y_{j,J}\cong T^J$ in $Y_{j,J}\cong D^2_j \x T^{J \sm j}$, and a trivial subbundle $\nu''$ of rank $n-k$. 
Spreading by the action of the torus $T^J$ we obtain a decomposition
\begin{equation}\label{nu-S-T}
  \nu(T^J\cdot (X\sm U)\subset\Z_P)=\nu(i\colon S^{2k-1} \x T^J \hra \Z_P)|_{T^J\cdot (X\sm U)} = \tau^{[m] \sm (I \cup J)} \oplus\eta^1\oplus \nu''.
\end{equation}
    
Now consider the normal bundle of $\6 U\x Y_{j,J}$ in~$\Z_P$, which is the normal bundle of the embedding $f'_j\colon S^{2k-2} \x D^2_j \x T^{J \sm j}
  \hookrightarrow\Z_P$.
We have $\6 U\x Y_{j,J}\subset V\subset\Z_P$, where $V \cong U\times Y_{j,J}\x D^{m+n-2k-l}$. Then we have a decomposition
\begin{equation}\label{nu-dU-Y}
  \nu(\6 U\x Y_{j,J}\subset\Z_P)=
  \nu(f'_j\colon S^{2k-2} \x D^2_j \x T^{J \sm j}
  \hookrightarrow\Z_P)=
  \wt{\eta}^1 \oplus \xi
  \oplus \tau^{[m] \sm (I \cup J)} ,
\end{equation}
where $\wt{\eta}^1=\nu(\6 U\x Y_{j,J} \subset U\x Y_{j,J})$ is a trivial one-dimensional bundle and $\xi$ is some bundle of rank~$n-k$.

From formulas~\eqref{nu-S-T} and~\eqref{nu-dU-Y} we obtain two decompositions of the normal bundles over two parts of the embedding~\eqref{g'}:
\begin{equation}\label{twoparts}
\begin{aligned}
  \nu(g')|_{T^J\cdot(X \sm U)} &=  \eta^1 \oplus \nu''\oplus
  \tau^{[m] \sm (I \cup J)},\\
  \nu(g')|_{\6 U \x Y_{j,J}} &= \wt{\eta}^1 \oplus \xi
   \oplus \tau^{[m] \sm (I \cup J)}.
\end{aligned}
\end{equation}
Moreover, the bundles $\nu''$ and $\xi$ are identified over the intersection $(T^J\cdot(X \sm U))\cap(\6 U \x Y_{j,J})=T^J\cdot\6U$ due to the natural identification of all other bundles:
\begin{multline*}
  \tau(\Z_P)|_{T^J \cdot\6U} = \tau(T^J \cdot U)|_{T^J\cdot\6U} \oplus 
  \nu\bigl(T^J\cdot (X\setminus U)\subset \Z_P\bigr)\big|_{T^J\cdot\6U} =\\= 
  \bigl( \tau^J \oplus \tau(\6 U) \oplus \wt{\eta}^1\bigr) \oplus 
  \bigl( \eta^1 \oplus \nu''\oplus \tau^{[m] \sm (I \cup J)} \bigr),
\end{multline*}
and at the same time
\begin{multline*}
  \tau(\Z_P)|_{T^J \cdot \6 U} =
  \tau(\6 U \x Y_{j,J})|_{T^J \cdot\6U} \oplus 
  \nu(\6 U \x Y_{j,J} \ss \Z_P)|_{T^J \cdot\6U}  =\\
  = \bigl(\tau(\6 U) \oplus \tau^J \oplus \eta^1\bigr) \oplus \bigl(\wt{\eta}^1 \oplus \xi \oplus \tau^{[m] \sm (I \cup J)}\bigr).
\end{multline*}
Note that the decomposition $\tau(Y_{j,J})|_{\6Y_{j,J}} = \tau^J \oplus \eta^1$ used in the last formula cannot be extended to the whole $Y_{j,J}\cong D^2\times T^{J\sm j}$, but we do not need this.

Since the bundles $\nu''$ and $\xi$ in~\eqref{twoparts} are identified over $T^J\cdot\6 U$ it remains only to smoothen the corners of the gyration $\mc{G}(S^{2k-1})$ at $S^{2k-2} \x S^1$ and to glue the bundles $\wt{\eta}^1$ and $\eta^1$ over $T^J\cdot\6 U$.

Consider the embedding of a collar neighbourhood $U_1\cong S^{2k-2} \x T^J \x [-\ve,0]$ of the boundary $T^J\cdot\6 U$ of the manifold $T^J\cdot(X\sm U)$:
\[
  i_1\colon S^{2k-2} \x T^J \x [-\ve,0]\to T^J\cdot(X\sm U), \q
  (s,t,r) \mapsto t\cdot \exp(r\wt\eta^1)\, i(s).
\]
where $\exp(r\wt\eta^1)\, i(s)$ is the shift of the point $i(s)\in\6 U\subset X$ by $r$ along the trajectory of the tangent vector field $\wt\eta^1$. 

Consider also the embedding of a collar neighbourhood $U_2\cong S^{2k-2} \x T^J \x [0, \ve]$ of the boundary $T^J\cdot\6 U$ of the manifold $\6 U \x Y_{j,J}$:
\[
  i_2\colon S^{2k-2} \x T^J \x [0, \ve]\to \6 U \x Y_{j,J}, \q
  (s,t,r) \mapsto t\cdot \exp(r\eta^1)\, i(s).
\]
where $\exp(r\eta^1)\, i(s)$ is the shift of the point $i(s)\in\6 U\subset X$ by $r$ along the trajectory of the tangent vector field $\eta^1$. 

Then the gyration $\mc{G}(S^{2k-1})$ can be decomposed as
\begin{equation}\label{gyrdec}
\begin{aligned}
   \mc{G}(S^{2k-1}) &\cong \bigl((S^{2k-1}\sm D^{2k-1})\x S^1 \bigr) \cup_{S^{2k-2} \x S^1 \x \{-\ve\}} S^{2k-2} \x S^1 \x [-\ve, 0] \\
   &\cup_{S^{2k-2} \x S^1 \x \{0\}} S^{2k-2} \x S^1 \x [0, \ve] \cup_{S^{2k-2} \x S^1 \x \{\ve\}} \bigl( S^{2k-2}\x D^2 \bigr),
    \end{aligned}
\end{equation}    
and now we regard $T^J \cdot \6 U = U_1 \cap U_2$ as separated from $T^J \cdot (X \sm U)$ and $\6 U \x Y_{j,J}$. In other words, after taking collars we can regard $T^J \cdot (X \sm U)$ and $\6 U \x Y_{j,J}$ as disjoint subsets in $g'(\mc{G}(S^{2k-2}))$. By gluing the segments $[-\ve,0]$ and $[0, \ve]$ we obtain $U_1 \cup_{T^J \cdot \6 U} U_2 \cong S^{2k-2} \x T^J \x [-\ve, \ve]$, but there is a corner at $r = 0$ that needs to be smoothened.

We replace the nonsmooth embedding~\eqref{g'} by a smooth one by shifting the image of $g'(\mc{G}(S^{2k-1})\x T^{J \sm j})$ along the trajectories of a smooth vector field that coincides with $\eta^1$ on $T^J\cdot (X\setminus U)$ and with $\wt\eta^1$ on $\6U\times Y_{j,J}$ (see Fig.~\ref{fig:GX-angles}). Explicitly, the new embedding $g\colon \mc{G}(S^{2k-1}) \x T^{J \sm j} \hra \Z_P$ is given by the following formulas on the parts of the union~\eqref{gyrdec}:  
\begin{align*}
  &  g|_{(S^{2k-1} \sm D^{2k-1}) \x T^J}(x,t) = \exp(\ve \eta^1)\,i(x,t) ,\\
  &  g|_{S^{2k-2}\times T^J\times[-\ve,\ve]}(s,t,r)  =  \exp\biggl(\ve\sqrt{1 - \frac{r + \ve}{2\ve}}\eta^1 + \ve\sqrt{\frac{r + \ve}{2\ve}}\wt{\eta}^1 \biggr)\, i(s,t),\\
  &  g|_{S^{2k-2} \x D^2 \x T^{J \sm j}}(s,y) = \exp(\ve\wt{\eta}^1) f'_j(s,y),
\end{align*}
where $x \in S^{2k-1} \sm D^{2k-1}$, $t \in T^J$, $s \in S^{2k-2}$ and $y \in D^2 \x T^{J \sm j}$. This is a smooth embedding, as $i$ and $f'_j$ agree on $T^J \cdot \6 U$.

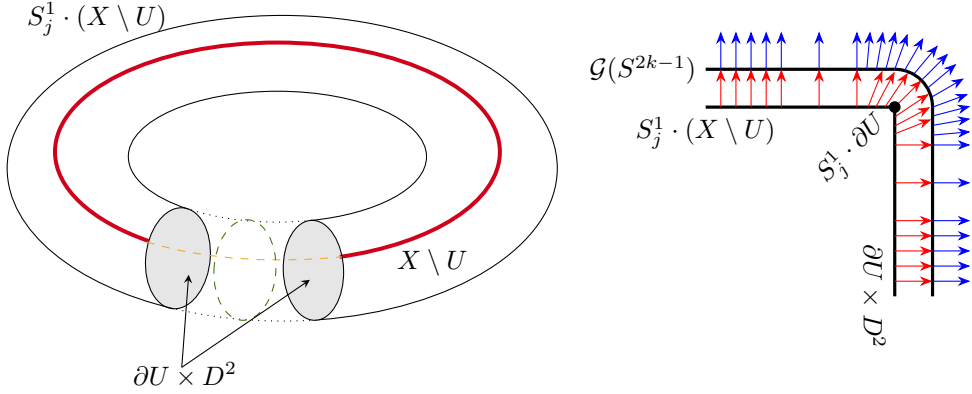
\begin{figure}[htbp]
    \centering
    \begin{minipage}{0.55\textwidth}
    \centering
\begin{tikzpicture}[x=1.5pt,y=1.5pt,scale=0.4]

\draw  [dotted] (208.1,139.39) .. controls (201.17,138.67) and (193.93,138.28) .. (186.5,138.28) .. controls (166.26,138.28) and (147.53,141.13) .. (132.23,145.98) ; 
\draw  [dotted] (211.09,76.99) .. controls (204.18,76.52) and (197.14,76.28) .. (190,76.28) .. controls (165.85,76.28) and (142.86,79.03) .. (121.98,83.99) ;   
\draw  [draw opacity=0] (121.33,84.05) .. controls (59.96,98.72) and (17,132.59) .. (17,172.03) .. controls (17,224.91) and (94.23,267.78) .. (189.5,267.78) .. controls (284.77,267.78) and (362,224.91) .. (362,172.03) .. controls (362,123.45) and (296.81,83.32) .. (212.34,77.11) -- (189.5,172.03) -- cycle ; 
\draw   (121.33,84.05) .. controls (59.96,98.72) and (17,132.59) .. (17,172.03) .. controls (17,224.91) and (94.23,267.78) .. (189.5,267.78) .. controls (284.77,267.78) and (362,224.91) .. (362,172.03) .. controls (362,123.45) and (296.81,83.32) .. (212.34,77.11) ;  

\draw  [draw opacity=0] (130.68,146.39) .. controls (107.81,153.86) and (93,165.81) .. (93,179.28) .. controls (93,201.92) and (134.86,220.28) .. (186.5,220.28) .. controls (238.14,220.28) and (280,201.92) .. (280,179.28) .. controls (280,159.69) and (248.66,143.30) .. (206.76,139.25) -- (186.5,179.28) -- cycle ; 
\draw   (130.68,146.39) .. controls (107.81,153.86) and (93,165.81) .. (93,179.28) .. controls (93,201.92) and (134.86,220.28) .. (186.5,220.28) .. controls (238.14,220.28) and (280,201.92) .. (280,179.28) .. controls (280,159.69) and (248.66,143.30) .. (206.76,139.25) ;  

\draw  [draw opacity=0][line width=1.5]  (105.56,126.48) .. controls (70.11,138.91) and (47,159.27) .. (47,182.28) .. controls (47,220.11) and (109.46,250.78) .. (186.5,250.78) .. controls (263.54,250.78) and (326,220.11) .. (326,182.28) .. controls (326,151.11) and (283.6,124.80) .. (225.57,116.50) -- (186.5,182.28) -- cycle ; 
\draw  [color={rgb, 255:red, 208; green, 2; blue, 27 }  ,draw opacity=1 ][line width=1.5]  (105.56,126.48) .. controls (70.11,138.91) and (47,159.27) .. (47,182.28) .. controls (47,220.11) and (109.46,250.78) .. (186.5,250.78) .. controls (263.54,250.78) and (326,220.11) .. (326,182.28) .. controls (326,151.11) and (283.6,124.80) .. (225.57,116.50) ;  

\draw  [fill={rgb, 255:red, 155; green, 155; blue, 155 }  ,fill opacity=0.25 ] (103.86,115.63) .. controls (105.04,133.11) and (115.06,147.28) .. (126.25,147.28) .. controls (137.43,147.28) and (145.54,133.11) .. (144.36,115.63) .. controls (143.19,98.16) and (133.16,83.99) .. (121.98,83.99) .. controls (110.79,83.99) and (102.68,98.16) .. (103.86,115.63) -- cycle ;

\draw  [fill={rgb, 255:red, 155; green, 155; blue, 155 }  ,fill opacity=0.25 ] (190.16,108.08) .. controls (189.7,125.38) and (197.79,139.39) .. (208.24,139.39) .. controls (218.69,139.39) and (227.53,125.38) .. (227.99,108.08) .. controls (228.45,90.79) and (220.36,76.77) .. (209.91,76.77) .. controls (199.46,76.77) and (190.62,90.79) .. (190.16,108.08) -- cycle ;

\draw  [draw opacity=0][dash pattern={on 0.84pt off 2.51pt}] (208.1,139.39) .. controls (201.17,138.67) and (193.93,138.28) .. (186.5,138.28) .. controls (166.26,138.28) and (147.53,141.13) .. (132.23,145.98) -- (186.5,179.78) -- cycle ; 

\draw  [draw opacity=0][dash pattern={on 0.84pt off 2.51pt}] (211.09,76.99) .. controls (204.18,76.52) and (197.14,76.28) .. (190,76.28) .. controls (165.85,76.28) and (142.86,79.03) .. (121.98,83.99) -- (190,172.28) -- cycle ; 

\draw [color={rgb, 255:red, 245; green, 166; blue, 35 }, dashed]  (104.44,125.92) .. controls (144.5,113.78) and (186.5,112.78) .. (227.02,116.50) ;

\draw  [color={rgb, 255:red, 65; green, 117; blue, 5 }, dashed]  (146.7,108.28) .. controls (147.36,125.68) and (156.59,139.78) .. (167.31,139.78) .. controls (178.03,139.78) and (186.19,125.68) .. (185.53,108.28) .. controls (184.86,90.88) and (175.63,76.77) .. (164.91,76.77) .. controls (154.19,76.77) and (146.03,90.88) .. (146.7,108.28) -- cycle ;

\draw[->,>=stealth] (126.5,47.28) -- (206.85,102.15) ;

\draw[->,>=stealth]  (126.5,47.28) -- (130.85,105.29) ;

\draw (260,122.78) node [anchor=north west][inner sep=0.75pt]   [align=left] {$\displaystyle X\setminus U \;$};
\draw (28,277.78) node [anchor=north west][inner sep=0.75pt]   [align=left] {$\displaystyle S^1_j \cdot ( X\setminus U)$};
\draw (94.5,52.28) node [anchor=north west][inner sep=0.75pt]   [align=left] {$\displaystyle \partial U\times D^{2}$};

\end{tikzpicture}   
    
    \end{minipage}
    \hfill
    \begin{minipage}{0.4\textwidth}
    \centering
\begin{tikzpicture}[>=Stealth]
    \draw[line width=1.2pt] (2.5,0) -- (2.5,2.5);
    \draw[line width=1.2pt] (0,2.5) -- (2.5,2.5);

    \draw[line width=1.2pt] (3,0) -- (3,2.5);
    \draw[line width=1.2pt] (0,3) -- (2.5,3);
    \draw[line width=1.2pt] (3,2.5) arc (0:90:0.5);

    \foreach \x in {0.2,0.4,0.6,0.8,1.0,1.5,2.0} {
        \draw[->, red] (\x, 2.5) -- (\x, 3);
        \draw[->, red] (2.5,\x) -- (3,\x);
    }

    
    \foreach \angle in {75,60,45} {
        \draw[->, red] ({2.5-cos(\angle)/2}, 2.5) -- ({2.5-cos(\angle)/2+0.5*sin(\angle/2)},{2.5+0.5*cos(\angle/3)});
    }
    
    \draw[->, red] (2.5, 2.5) -- ({2.5+0.5*cos(45)},{2.5+0.5*sin(45)});
    
    
    \foreach \angle in {75,60,45} {
        \draw[->, red] (2.5, {2.5-cos(\angle)/2}) -- ({2.5+0.5*cos(\angle/3)},{2.5-cos(\angle)/2+0.5*sin(\angle/2)});
    }

    \foreach \x in {0.2,0.4,0.6,0.8,1.0,1.5,2.0} {
        \draw[->, blue] (\x, 3) -- (\x, 3.5);
        \draw[->, blue] (3,\x) -- (3.5,\x);
    }

    
    \draw[->, blue] ({2.5+0.5*cos(45)},{2.5+0.5*sin(45)}) -- ({2.5+cos(45)},{2.5+sin(45)});
    
    \foreach \angle in {15,30,45} {
        \draw[->, blue] ({2+sin(\angle/2)}, 3) -- ({2+sin(\angle/2)+sin(\angle/3)/2}, 3.5);
        \draw[->, blue] ({2.5+sin(\angle-15)/2},{2.5+cos(\angle-15)/2}) -- ({2.5+sin(\angle-15)/2+sin(\angle*2/3+10)/2},{2.5+cos(\angle-15)/2+cos(\angle*2/3+10)/2});
        
        \draw[->, blue] (3, {2+sin(\angle/2)}) -- (3.5,{2+sin(\angle/2)+sin(\angle/3)/2});
        \draw[->, blue] ({2.5+cos(\angle-15)/2},{2.5+sin(\angle-15)/2}) -- ({2.5+cos(\angle-15)/2+cos(\angle*2/3+10)/2},{2.5+sin(\angle-15)/2+sin(\angle*2/3+10)/2});
    }
    
    \filldraw (2.5,2.5) circle (2pt);

    \node[below] at (0, 2.5) {$S^{1}_j \cdot (X\setminus U)$}; 
    \node[left] at (0, 3) {$\mathcal{G}(S^{2k-1})$}; 
    \node[below][rotate=-90] at (2.5, 0) {$\partial U \times D^2$};  
    \node[left][rotate=45] at (2.5, 2.5) {$S^1_j \cdot \partial U \;$}; 
\end{tikzpicture}    

    \end{minipage}
    \caption{The embedding $\mc{G}(S^{2k-1}) \hra \Z_P$ with corners (left) and its smoothing (right).}
    \label{fig:GX-angles}
\end{figure}
    
The trivialization of the normal bundle $\nu(g)$ is described similarly in terms of the trivializations of $\nu(i)$ and $\nu(f'_j)$:
\begin{align*}
  &\nu(g)|_{T^J \cdot (X \sm U)}(x,t) = \nu(i)(x,t),\q \nu(g)|_{\6 U \x Y_{j,J}}(s,y) = \nu(f'_j)(s,y),\\
  &\nu(g)|_{U_1 \cup U_2}(s,t,r) = (\tau^{[m] \sm (I \cup J)} \oplus \nu'')(s,t) \oplus \Bigl(\sqrt{1 - \frac{r + \ve}{2\ve}} \cdot \eta^1 + \sqrt{\frac{r + \ve}{2\ve}} \cdot\wt{\eta}^1\Bigr).
\end{align*}
These formulas agree on the intersection, as they contain the common summand $\tau^{[m] \sm (I \cup J)}$ of $\nu'$ and $\nu(f'_j)$ and the bundles $\nu''$ and $\xi$ are identified over $T^J \cdot \6 U$.
    
Note now that the embedding~\eqref{bordism} gives a bordism between $i$ and $g$, hence in the group $H_{2k-1+l}(\Z_P)$ we have
\[
  g_*[\mc{G}(S^{2k-1}) \x T^{J \sm j}] = i_*[S^{2k-1} \x T^J] \pm [\6 (D^{2k-1 } \x D^2 \x T^{J \sm j})] = i_*[S^{2k-1} \x T^J].
\]
Performing the same construction iteratively for all elements $j \in J$, we obtain the required embedding $g\colon\mc{G}^{l}(S^{2k-1}) \hra \Z_P$.
\end{proof}

\section{Starshaped spheres (simplicial spheres arising from fans)}
If a simplicial sphere $\K$ is not polytopal, then the corresponding topological moment-angle manifold $\Z_\K$ cannot be given by an intersection of quadrics~\eqref{intqu}. Consequently, the techniques developed in Section~\ref{seciq} for representing homology classes in~$\Z_\K$ by submanifolds with trivial normal bundles cannot be applied.

Nevertheless, in the case when the simplicial sphere $\K$ arises from a complete simplicial fan (i.\,e., is a \emph{starshaped sphere}), the moment-angle manifold $\Z_\K$ can be endowed with a smooth structure via the quotient construction by the exponential action~\cite{pa-25}. In this case, a number of results on diffeomorphisms of moment-angle manifolds and connected sums of products of spheres can be extended to nonpolytopal spheres. In particular, we prove that the moment-angle manifolds corresponding to the classical 3-dimensional nonpolytopal spheres, the Barnette sphere and the Br\"uckner sphere, are diffeomorphic to connected sums of products of spheres.

Let $\K$ be a simplicial complex on the set $[m]$ and $\mathrm A=\{a_1,\ldots,a_m\}$ be a spanning vector configuration in $W^*\cong\R^n$. For a subset $I\subset[m]$ denote by $\cone\mathrm A_I$ the cone in $W^*$ generated by the vectors $\{a_i\colon i\in I\}$ and denote by $\relint\cone\mathrm A_I$ the relative interior of the cone. We say that the data $\{\K,\mathrm A\}$ \emph{defines a fan} (or $\{\K,\mathrm A\}$ is a \emph{triangulated configuration}), if
\[
  \Sigma_{\K,\mathrm A}=\bigl\{\cone \mathrm A_I\colon I\in \K\bigr\}
\]
is a simplicial fan in~$W^*$. This means that the following two conditions are satisfied:
\begin{itemize}
\item[(a)] the vectors $\mathrm A_I=\{a_i\colon i\in I\}$ are linearly independent for every  $I\in\K$ ;

\item[(b)] $(\relint\cone\mathrm A_I)\cap(\relint\cone\mathrm A_J)=\varnothing$ for $I,J\in\K$, $I\ne J$.
\end{itemize}

If $\Sigma_{\K,\mathrm A}$ is a \emph{complete} simplicial fan, i.\,e. $\bigcup_{I\in\K}\cone \mathrm A_I=W^*$, then $\K$ is a simplicial sphere admitting a \emph{starshaped} realization in~$W^*$ (a \emph{starshaped sphere}). If, moreover, $\Sigma_{\K,\mathrm A}=\Sigma_P$ is the \emph{normal fan} of a convex simple polytope~\eqref{ptope}, then $\K$ is a polytopal sphere, i.\,e. $\K=\K_P$, see~\eqref{nerve}. 

Let $\Gamma=\{\gamma_1,\ldots,\gamma_m\}$ be the Gale dual to $\mathrm A$ vector configuration in $V\cong\R^{m-n}$, see Construction~\ref{galed}. Define the \emph{exponential action} of the group $V$ on~$\C^m$:
\begin{equation}\label{exp-action}
    V \x \C^m \to \C^m,\quad (v,z) \mapsto v \cdot z = (e^{\langle \gamma_1, v \rangle} z_1, \ldots, e^{\langle \gamma_m, v \rangle} z_m),
\end{equation}
Define the open subset in $\C^m$ (the complement of an arrangement of coordinate subspaces)
\[
  U(\K) = \C^m \sm \bigcup_{J=\{j_1,\ldots,j_k\} \notin \K} 
  \{z \in \C^m\colon z_{j_1}=\cdots=z_{j_k}=0\}.
\]

\begin{thm}[{\cite[Theorem 6.5.2]{BP}}]\label{zkquo}\ 
\begin{enumerate}
\item[(1)] If the data $\{\K, \mathrm A\}$ defines a fan $\Sigma$, then the exponential action of $V$ on $U(\K)$ is free and proper and the quotient space $U(\K)/ V$ is a smooth manifold of dimension $m+n$;
       
\item[(2)] if the fan $\Sigma$ is complete then $U(\K)/ V$ is homeomorphic to the moment-angle manifold $\Z_\K$.
\end{enumerate}
\end{thm}

Therefore, the moment-angle manifold $\Z_\K$ corresponding to a starshaped sphere $\K$ (arising from a complete simplicial fan~$\Sigma$) is endowed with a smooth structure. If, moreover, $\Sigma$ is the normal fan of a simple polytope, then this smooth structure is equivalent to the smooth structure on $\Z_P$ coming from the intersection of quadrics:

\begin{thm}[{\cite[Theorem~8.4]{pa-25}}]
If $\Sigma=\{\cone \mathrm A_I\colon I\in \K_P\}$ is the normal fan of a simple polytope~\eqref{ptope}, then the quotient space $U(\K_P)/ V$ is diffeomorphic to the intersection of quadrics~\eqref{intqu}.
\end{thm}

\begin{prop}\label{Z-K-section}
Let $\K$ be a starshaped $(n-1)$-sphere with $m$ vertices, Then there exists a smooth embedding $\Z_\K \hra \C^m$ with trivial normal bundle.
\end{prop}

\begin{proof}
Since the group $V \cong \mb{R}^{m-n}$ is contractible, the principal $V$-bundle $U(\K) \to U(\K)/V\cong\Z_\K$ is trivial. Hence there exists a smooth section $s \colon \Z_\K \hra U(\K) \ss \C^m$, which gives the required embedding.
\end{proof}

For starshaped spheres the following analogue of Lemma~\ref{normal-mf-submanifold} holds.

\begin{lemma}\label{normal-mf-starshaped}
Suppose the data $\{\K,\mathrm A\}$ defines a complete simplicial fan, and let $I$ be a missing face in $\K$ with $|I|=k$. 
Then there exists a smooth embedding $S^{2k-1}\x T^{[m]\sm I}\hra \Z_\K$ with trivial normal bundle, whose restriction to $S^{2k-1}\times T^J$ for any $J\subset[m]\sm I$ is a smooth submanifold representing the homology class in $H_{2k-1+|J|}(\Z_\mathcal K)$ corresponding to the missing face~$I$.
\end{lemma}

\begin{proof}
To construct a smooth embedding, we need to consider two cases.

\textit{Case 1:} the vectors $\mathrm A_I$ are linearly dependent.
For each $i \in I$ the vectors $\mathrm A_{I \sm i}$ are linearly independent since $I \sm \{i\}\in \K$. Hence, the vectors $\mathrm A_I$ satisfy exactly one linear relation $\sum_{i \in I} \lambda_i a_i = 0$.
Let $\widetilde W^*=\R\langle\mathrm A_I\rangle$ be the subspace in $W^*$ spanned by the vectors $\mathrm A_I$, then $\dim \widetilde W^*=k-1$. Choose a projection $p\colon W^*\to\widetilde W^*$ and consider the configuration $\widetilde{\mathrm A}=p(\mathrm A)$ of $m$ vectors in~$\widetilde W^*$; note that $\widetilde a_i=a_i$ for $i\in I$. Consider the short exact sequence
\[
  0\longrightarrow\widetilde V\longrightarrow\R^m\stackrel{\widetilde A}\longrightarrow 
  \widetilde W^*  \longrightarrow 0,
\]
where $\widetilde A(e_i)=\widetilde a_i$ and $\widetilde V\cong\R^{m+1-k}$ is the space of linear relations among the vectors~$\widetilde{\mathrm A}$. We have $\widetilde A=p\circ A$, which implies $\ker A\subset\ker\widetilde A$ and hence $\wt{V}=V \oplus V'$, where $V$ is the space of linear relations on $\mathrm A$ and $V'$ is a complementary subspace with $\dim V' = n+1-k$. Let $\K_{I,[m]}$ be the full subcomplex of $\K$ corresponding to the missing face~$I$, regarded as a simplicial complex on~$[m]$, i.\,e. $[m]\setminus I$ is the set of its ghost vertices. Then the data $\{\K_{I,[m]},\widetilde{\mathrm A}\}$ define a complete simplicial fan in~$\widetilde W^*$. By Theorem~\ref{zkquo}, $\Z_\K\cong U(\K)/V$ and $\Z_{\K_{I,[m]}} \cong U(\K_{I,[m]})/\wt{V}$. Moreover, $U(\K_{I,[m]})$ is a Zariski open subset of $U(\K)$, i.\,e. $U(\K_{I,[m]})=U(\K)\setminus R$, where $R$ is a closed subset (the arrangement of coordinate subspaces corresponding to $J\in\K$, $J\notin\K_I$). We have
\[
  S^{2k-1}\x T^{[m]\sm I} \cong \Z_{\K_{I,[m]}} \cong U(\K_{I,[m]}) / \wt{V} = \bigl(U(\K)\setminus R\bigr)/(V\oplus V')\cong
  \bigl(\Z_\K\setminus R'\bigr)/V',
\]
where $R'\cong R/V$ is a closed subset of $\Z_\K$. Since the principal $V'$-bundle $(\Z_\K \sm R')\to S^{2k-1}\x T^{[m]\sm I}$ is trivial, its smooth section provides the required embedding $S^{2k-1}\x T^{[m]\sm I}\hra \Z_\K$ with trivial normal bundle.

\smallskip

\textit{Case 2}: the vectors $\mathrm A_I$ are linearly independent. As in Case~1 we define $\wt W^*=\R\langle\mathrm A_I\rangle$, the configuration $\wt{\mathrm A}=p(\mathrm A)$, the space of linear relations $\wt{V}$ and the decomposition $\wt{V}=V \oplus V'$, but now $\dim \wt W^*=k$, $\dim \wt V^*=m-k$ and $\dim V' = n-k$. The data $\{\K_{I,[m]},\wt{\mathrm A}\}$ define an \textit{incomplete} simplicial fan in the space $\wt W^*$ so the action of $\wt{V}$ on $U(\K_{I,[m]})$ is proper, but the quotient $U(\K_{I,[m]})/\wt{V}$ is noncompact. Denote by $\wt{\Gamma}$ the Gale dual to $\wt{\mathrm A}$ configuration. Since $\wt{\mathrm A}_I$ is a basis of $\wt{W}^*$ then $\wt{\Gamma}_{\wh{I}}$ is a basis of $\wt{V}^*$ (Proposition~\ref{dspan}). We have
\[
  U(\K_{I,[m]}) / \wt{V} = \big( (\mb{C}^I \sm \{0\}) \x (\mb{C} \sm \{0\})^{[m]\sm I} \big) / \wt{V} \cong (\mb{C}^I \sm \{0\}) \x T^{[m]\sm I},
\]
since for any point $(z_1, \ldots, z_m) \in U(\K_{I,[m]})$ there exists $v \in \wt{V}$ such that $\langle \gamma_j, v \rangle = -\ln{|z_j|}$ for all $j \notin I$.
Further,
\[
(\mb{C}^I \sm \{0\}) \x T^{[m]\sm I} \cong \mb{R} \x S^{2k-1}\x T^{[m]\sm I} \cong \mb{R} \x \Z_{\K_{I,[m]}}.
\]
Accordingly, defining $R$ and $R'$ as in Case~1, we obtain
\[
  \mb{R} \x S^{2k-1}\x T^{[m]\sm I} \cong U(\K_{I,[m]}) / \wt{V} = \bigl(U(\K)\setminus R\bigr)/(V\oplus V') = \bigl(\Z_\K\setminus R'\bigr)/V'.
\]
Then the composition of the principal $V'$-bundle and the trivial one-dimensional bundle
\[
  (\Z_\K \sm R')\longrightarrow \mb{R} \x S^{2k-1}\x T^{[m]\sm I} \longrightarrow S^{2k-1}\x T^{[m]\sm I},
\]
is a trivial bundle, and its smooth section provides the required embedding $S^{2k-1}\x T^{[m]\sm I}\hra \Z_\K$ with trivial normal bundle.
\end{proof}

The following analogue of Lemma~\ref{normal_bundles} holds.

\begin{lemma}\label{normal_bundles-starshaped}
Let the data $\{\K,\mathrm A\}$ define a complete $n$-dimensional simplicial fan, and let $I$ be a missing face of $\K$ with $|I|=k\leq n-1$.
Consider the smooth embedding $i\colon S^{2k-1}\times T^J \hra \Z_\K$ with trivial normal bundle from Lemma~\ref{normal-mf-starshaped}, where $J\ss [m] \sm I$, $|J|=l$. 
Then there exists a smooth embedding of the $l$-fold gyration $g\colon\mc{G}^{l}(S^{2k-1}) \hra \Z_\K$ with trivial normal bundle such that the embeddings $i$ and $g$ are bordant. Therefore, the homology class in $H_{2k-1+l}(\Z_\K)$ corresponding to the missing face $I$ is represented by an embedded sphere $g(\mc{G}^{l}(S^{2k-1}))\cong S^{2k-1+l}$.
\end{lemma}

\begin{proof}
The proof repeats the proof of Lemma~\ref{normal_bundles} except for the construction of the tori $Y_j\cong D^2\times T^{m-1}$, $j\in J$ which we give below.

Consider the embedding $\Z_\K\hookrightarrow\C^m$ from Proposition~\ref{Z-K-section} and the submanifold $X\cong S^{2k-1}\times T^{[m]\setminus I}$ in $\Z_\K$ constructed in Lemma~\ref{normal-mf-starshaped}. The images under the moment maps $\mu(\Z_\K)$ and $\mu(X)$ are smooth manifolds with corners with $\mu(X)\cong\Delta^{k-1}$ and
$
  \oname{codim}(\mu(X) \ss \mu(\Z_\K)) = n-k+1 \geq 2.
$
Choose points $x_j \in \oname{relint}\mu(X)$ and $y_j \in \oname{relint}\mu(\Z_\K \cap \{z_j = 0\})$ for $j \in J$ and pairwise disjoint and non-self-intersecting smooth curves $\varphi_j\colon [0,1] \to \mu(\Z_\K)$, $\varphi_j(0) = x_j$, $\varphi_j(1) = y_j$. (It is possible since $\mu(\Z_\K)$ is connected and $\mu(X)$ has codimension $\geq 2$ in $\mu(\Z_\K)$.)
Then the preimage $Y_j := \mu^{-1}(\varphi_j([0,1])) \cong D^2 \x T^{m-1}$ is a smooth submanifold with boundary in $\Z_\K$. Since $\mu(X)\cap\mu(Y_j)=x_j$ we obtain $X\cap Y_j=\mu^{-1}(x_j) \cong T^m$. 
%
\end{proof}

Now we can establish the diffeomorphism in case~(b) of Theorem~\ref{chordal4iff} under the additional assumption on~$\K^1$:

\begin{thm}\label{connected-starshaped}
Let $\K$ be a 3-dimensional starshaped sphere on $[m]$ such that $\K^1$ is a chordal graph and there is a vertex $v \in [m]$ adjacent to all other vertices. Then $\Z_\K$ is diffeomorphic to a connected sum of products of pairs of spheres.
\end{thm}

\begin{proof}
The proof is based on the application of Theorem~\ref{h-cob} and follows the same scheme as the proof of Theorem~\ref{diffeo4}.

Define $\mathcal Q = (\bm{X}, \bm{A})^\K$ as the polyhedral product corresponding to the collection of pairs $(X_v, A_v) = (D^3_+, S^2_+)$ and $(X_i, A_i) = (D^2, S^1)$ for $i \neq v$.
Since $\K$ arises from a complete simplicial fan, the quotient construction by the exponential action shows that $\mc Q$ is a smooth manifold with boundary $\Z_\K$. 
Since the pair $(D^3_+, S^2_+)$ is contractible we obtain $\mathcal Q \simeq (D^2,S^1)^{\K_{[m] \sm v}} = \Z_{\K_{[m] \sm v}}$.

Next we provide a collection of embedded smooth submanifolds $X_j \ss \Z_\K$ with trivial normal bundles, required by Theorem~\ref{h-cob}.

Since $\K^1$ is a chordal graph then the groups $\H_{0,*}(\K_{[m] \sm v})$ and $\H_{1,*}(\K_{[m] \sm v})$ are generated by classes corresponding to missing one-dimensional and two-dimensional faces of~$\K_{[m] \sm v}$, respectively.

We prove that $\H_{k,*}(\K_{[m] \sm v}) = 0$ for $k \geq 2$. Indeed, $\H_{k,*}(\K_{[m] \sm v}) = 0$ for $k \geq 3$, since $\K$ is a 3-dimensional sphere. Furthermore, suppose that for some $I \ss [m] \sm v$ we have $\H_{2,I}(\K_{[m] \sm v}) = \H_{2,I}(\K) \neq 0$, then by Poincar\'e duality~\eqref{aldua} we get $\H^{0,[m] \sm I}(\K) \cong \H_{2,I}(\K)$. However $v \in [m] \sm I$ so the full subcomplex $\K_{[m] \sm I}$ is connected and we obtain $\H^{0,[m] \sm I}(\K) = 0$ --- a contradiction. 
    
It follows that 
\[
  H_*(\mc Q) \cong \H_{*,*}(\K_{[m] \sm v}) \cong \mb{Z} \langle a^0_\alpha \colon \alpha \in A \rangle \oplus \mb{Z}\langle b^1_\beta \colon \beta \in B \rangle,
\]
where the classes $a^0_\alpha$ correspond to missing edges, and the classes $b^1_\beta$ correspond to missing triangles. All of them are represented by spheres with trivial normal bundles by Lemma~\ref{normal-mf-starshaped}.
    
Therefore, we can apply Theorem~\ref{h-cob} which gives us that $\Z_\K=\6\mathcal Q$ is diffeomorphic to a connected sum of products of spheres.
\end{proof}

Finally, we can extend Theorem~\ref{diffeo4} (establishing the diffeomorphism in case~(c) of Theorem~\ref{chordal4iff}) to the case of starshaped spheres:

\begin{thm}\label{diffeo4-starshaped}
Let $\K$ be a 3-dimensional starshaped simplicial sphere with only
two missing edges which are not adjacent to each other. Then there is a diffeomorphism $\Z_\K \cong M_1\#\cdots\# M_k$, where each $M_i$ is a product of spheres, and one of the $M_i$ is a product of three spheres.
\end{thm}

\begin{proof}
We can repeat the proof of Theorem~\ref{diffeo4} (taking into account the remarks on the smoothness of $\mc Q$ from the proof of Theorem~\ref{connected-starshaped}).
A 3-dimensional sphere satisfying the assumptions of the theorem has $m \geq 8$ vertices, which allows us to apply Lemma~\ref{by-Iriye}. Instead of Lemma~\ref{normal_bundles} we use Lemma~\ref{normal_bundles-starshaped}.
\end{proof}

\section{Neighbourly simplicial spheres}
A simplicial sphere $\K$ of dimension $n-1$ on the vertex set $[m]$ is called \textit{neighbourly} if $\K$ contains all $\lfloor\frac{n}{2}\rfloor$-element subsets.

Theorem~\ref{connected-starshaped} implies that for 3-dimensional neighbourly starshaped spheres~$\K$ the moment-angle manifold $\Z_\K$ is diffeomorphic to a connected sum of products of spheres. In fact, the same holds for neighbourly starshaped spheres of arbitrary odd dimension:

\begin{thm}\label{odd-dim-connected-starshaped}
Let $\K$ be a starshaped neighbourly simplicial sphere of odd dimension $n-1$ with $\K\ne\partial\Delta^n$. Then there is a diffeomorphism $\Z_\K \cong M_1\#\cdots\# M_k$, where each $M_i$ is a product of two spheres.
\end{thm}

\begin{proof}
Set $q=\frac{n-2}{2}$. As in the proof of Theorem~\ref{connected-starshaped} we construct a manifold $\mathcal Q$ with boundary $\partial\mathcal Q\cong\Z_\K$ and $\mathcal Q \simeq \Z_{\K_{[m] \sm \{1\}}}$.

Since $\K$ is neighbourly, $\H_{k,I}=\widetilde H_k(\K_I) = 0$ for $k<q$ and $\H_{q,I}$ is generated by the classes corresponding to missing faces for any $I\subset[m]$. By Poincar\'e duality we also have $\H_{k,I} = 0$ for $q < k < n-1$, i.\,e. the groups $\H_{k,I}$ can be nonzero only for $k = q$ and $k = n-1$. 

Since the embedding of the boundary $\partial\mathcal Q\hookrightarrow\mathcal Q$ induces a surjection on homology groups, we obtain
\[
  \widetilde H_*(\mathcal Q)=\H_{q,*}(\K_{[m] \sm \{1\}})=
  \mathbb Z \langle b_\alpha\colon \alpha\in A\rangle\cong
  \widetilde H_*\Bigl( \bigvee_{\alpha \in A} S^{n_\alpha}\Bigr),
\]
where $\{b_\alpha\colon \alpha\in A\}$ is a set of classes corresponding to missing faces. These classes are represented by embedded spheres with trivial normal bundles by Lemma~\ref{normal_bundles-starshaped}. 
Therefore, we can apply Theorem~\ref{h-cob} which gives us that $\Z_\K=\6\mc Q$ is diffeomorphic to a connected sum of products of spheres.
\end{proof}

For even-dimensional neighbourly simplicial spheres, there is only a cohomological result under an additional torsion-free assumption:

\begin{prop}
Let $\K\ne\partial\Delta^n$ be a neighbourly simplicial sphere of even dimension $n-1$ and suppose that the group $\H^{\frac{n-1}{2},*}(\K)$ is torsion-free. Then there is a ring isomorphism $H^*(\Z_\K) \cong H^*(M_1\#\cdots\# M_k)$, where each $M_i$ is a product of two spheres.
\end{prop}

\begin{proof}
Set $q =  \frac{n-3}{2}$. By neighbourliness we have $\H^{k,*}(\K) = 0$ for $k < q$ and by Poincar\'e duality $\H^{k,*}(\K) = 0$ for $q+1 < k < n-1$. For dimensional reasons, all nontrivial multiplications have the form 
\[
  \H^{q, *}(\K) \ox \H^{q+1, *}(\K) \longrightarrow \H^{n-1, *}(\K)=\mathbb Z.
\]
The torsion of $\H^{q, *}(\K)$ is zero, since it equals the torsion of $\H_{q-1, *}(\K) = 0$, and $\H^{q+1, *}(\K)$ is torsion-free by assumption. This implies the required result.
\end{proof}

\begin{rem}
The condition that the group $\H^{\frac{n-1}{2},*}(\K) \cong \H_{\frac{n-3}{2},*}(\K)$ is torsion-free is essential. An example of a neighbourly even-dimensional sphere $\K$ containing a complex with torsion in homology (for instance, a minimal triangulation of $\mb{R}P^2$) as a full subcomplex would provide a counterexample to the conjecture of Bosio and Meersseman~\cite[p.~115]{bo-me06}.
\end{rem}

\begin{expl}[Barnette and Br\"uckner spheres]\label{bbexa}
The Barnette sphere is a well-known example of a three-dimensional simplicial sphere that is not polytopal. Its construction (see, e.g.,~\cite[Construction 2.5.5]{BP}) starts with a configuration of three tetrahedra $abde$, $acdf$ and $bcef$ in Fig.~\ref{fig:Barnette-Bruckner}, left. Then one adds the vertex $p$ and the tetrahedra $pabc$, $pdef$, $pabd$, $pbde$, $pbce$, $pcef$, $pacf$ and $padf$ inside, as well as the vertex $p'$ and the tetrahedra $p'abc$, $p'def$, $p'abe$, $p'ade$, $p'acd$, $p'cdf$, $p'bcf$ and $p'bef$ outside, as in Fig.~\ref{fig:Barnette-Bruckner}, right. The result is a simplicial triangulation of a 3-dimensional sphere with $8$ vertices and $19$ tetrahedra, which we denote by~$\K_\Ba$.
The Br\"uckner sphere is obtained from $\K_\Ba$ by a bistellar flip along the edge $pp'$, as the result of which the two tetrahedra $pabc$ and $p'abc$ are replaced by the three tetrahedra $pp'ab$, $pp'bc$ and $pp'ac$ (see Fig.~\ref{fig:Barnette-Bruckner}). We denote the corresponding simplicial complex by $\K_\Br$; it is neighbourly, has $8$ vertices and $20$ tetrahedra. The Barnette sphere and the Br\"uckner sphere are nonpolytopal, but they are starshaped (see~\cite[Example~2.5.8]{BP}). Therefore, the corresponding moment-angle complexes $\Z_{\K_\Ba}$ and $\Z_{\K_\Br}$ are smooth manifolds.
\begin{figure}[htbp]
    \centering

\tikzset{every picture/.style={line width=0.75pt, x=1pt,y=1pt,scale=0.7, baseline=(current bounding box.center)}}

\hfill
\begin{tikzpicture}

\coordinate (A) at (50,30);   
\coordinate (B) at (100,30); 
\coordinate (C) at (75,70);  

\coordinate (D) at (0,0);    
\coordinate (E) at (150, 0);     
\coordinate (F) at (75, 120);    

\draw (A) -- (B) -- (C) -- cycle;
\draw (D) -- (E) -- (F) -- cycle;

\draw (D) -- (A) -- (E);
\draw (E) -- (B) -- (F);
\draw (F) -- (C) -- (D);

\draw[dashed] (A) -- (F);
\draw[dashed] (B) -- (D);
\draw[dashed] (C) -- (E);

\begin{scope}[every node/.style={font=\tiny, inner sep=1pt}]
\node[below] at (A) {$a$};
\node[below] at (B) {$b$};
\node[above right] at (C) {$c$};
\node[left] at (D) {$d$};
\node[right] at (E) {$e$};
\node[above] at (F) {$f$};
\end{scope}

\end{tikzpicture}
\hfill
\begin{tikzpicture}

\coordinate (A) at (50,70);   
\coordinate (B) at (105,75); 
\coordinate (C) at (100, 50);  

\coordinate (D) at (0,0);    
\coordinate (E) at (170, 0);     
\coordinate (F) at (85, -30);   

\coordinate (P') at (75, 145);    
\coordinate (P) at (75, 20);    

\draw[line width=0.5pt, opacity=0.75] (P) -- (A) -- (P') -- (B) -- cycle;
\draw[line width=0.5pt, opacity=0.75] (P) -- (C) -- (P');
\draw[dashed, color=red] (P) -- (P');
\fill[red, fill opacity=0.2] (A) -- (B) -- (C) -- cycle;

\draw (D) -- (A) -- (B) -- (E) -- (F) -- cycle;
\draw (D) -- (P') -- (E) -- cycle;
\draw (F) -- (A) -- (C) -- cycle;
\draw (E) -- (C) -- (B);

\draw[dashed] (B) -- (D);
\draw[line width=0.5pt] (P') -- (F);

\begin{scope}[every node/.style={font=\tiny, inner sep=1pt}]
\node[above left] at (A) {$a$};
\node[above right] at (B) {$c$};
\node[right] at (C) {$b$};
\node[left] at (D) {$d$};
\node[right] at (E) {$f$};
\node[below right] at (F) {$e$};
\node[above] at (P') {$p'$};
\node[below] at (P) {$p$};
\end{scope}

\end{tikzpicture}
\hfill
\hfill
\caption{Starting point of the construction of the Barnette sphere (left).
The Barnette and Br\"uckner spheres differ by one bistellar flip on $\{abcpp'\}$ (right).}
\label{fig:Barnette-Bruckner}
\end{figure}
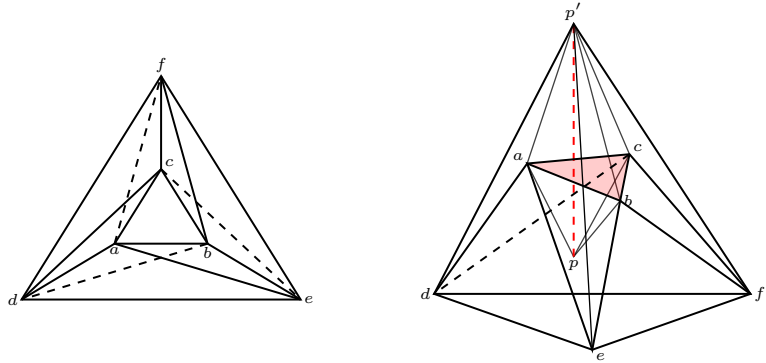
\end{expl}

\begin{prop}\label{Barnette-Bruckner} 
The moment-angle manifolds $\Z_{\K_{\Ba}}$ and $\Z_{\K_{\Br}}$ corresponding to the Barnette and Br\"uckner spheres are diffeomorphic to connected sums of products of pairs of spheres, namely,
\begin{align*}
    &\Z_{\K_\Ba} \cong (S^3 \x S^9) \# (S^5 \x S^7)^{\#12} \# (S^6 \x S^6)^{\#12}\\
    &\Z_{\K_\Br} \cong (S^5 \x S^7)^{\#16} \# (S^6 \x S^6)^{\#15} \,.
\end{align*}
\end{prop}

\begin{proof}
The spheres $\K_{\Ba}$ and $\K_{\Br}$ satisfy the assumptions of Theorem~\ref{connected-starshaped} since they have chordal $1$-skeleton and the vertex $f$ is adjacent to all other vertices, see Fig.~\ref{fig:Barnette-Bruckner}.
    
It remains only to count the number of spheres in each dimension, for this it suffices to find $\H_{0,*}$ and $\H_{1,*}$. The sphere $\K_{\Ba}$ has exactly one missing edge, hence $\H_{0,*}(\K_\Ba) \cong \mb{Z}$ which corresponds to the summand $S^3 \x S^9$. The sphere $\K_{\Br}$ is neighbourly and therefore $\H_{0,*}(\K_\Br) = 0$.
To compute $\H_{1,*}$ we list the missing faces for the Barnette and Br\"uckner spheres:
\begin{align*}
  \MF(\K_\Ba) = \{&pp', abf, ace, bcd, cde, bdf, aef, bdp', cep', afp', cdp, aep, bfp\},\\
  \MF(\K_\Br) = \{&abf, ace, bcd, cde, bdf, aef, bdp', cep', afp', 
  cdp, aep, bfp, \\
  &abc, pp'd, pp'e, pp'f\}.
\end{align*}
By Poincar\'e duality it suffices to compute $\H_{1,I}$ for $|I| = 3,4$. For $|I| = 3$ the Barnette sphere has $12$ missing faces, which corresponds to the summands $(S^5 \x S^7)^{\#12}$, and the Br\"uckner sphere has $16$ missing faces, which corresponds to the summands $(S^5 \x S^7)^{\#16}$.
    
Now we list all full subcomplexes with nontrivial first homology for $|I| = 4$. This is easy to do after noting that such a subcomplex must contain at least two missing faces. 
For the Barnette sphere we have $\H_{1,I}(\K_\Ba) \cong \mb{Z}$ for
\begin{align*} 
  I \in \{
  &abdf, pp'ce, abef, pp'cd, bcdf, pp'ae, bcde, pp'af, acde, pp'bf, acef, pp'bd,\\
  &abpf, cp'de, bcpd, ap'ef, acpe, bp'df, apef, bcp'd, bpdf, acp'e, cpde, abp'f
  \}
\end{align*}
and $\H_{1,I}(\K_\Ba) = 0$ otherwise; this corresponds to the summands $(S^6 \x S^6)^{\#12}$. 
For the Br\"uckner sphere we have $\H_{1,I}(\K_\Br) \cong \mb{Z}$ for
\begin{align*}
  I \in \{
  &abdf, pp'ce, abef, pp'cd, bcdf, pp'ae, bcde, pp'af, acde, pp'bf, acef, pp'bd,\\
  &abpf, cp'de, bcpd, ap'ef, acpe, bp'df, apef, bcp'd, bpdf, acp'e, cpde, abp'f,\\
  &abcf, pp'de, abcd, pp'ef, abce, pp'df, 
  \}
\end{align*}
and $\H_{1,I}(\K_\Br) = 0$ otherwise; this corresponds to the summands $(S^6 \x S^6)^{\#15}$.
\end{proof}


\begin{thebibliography}{FCMW}

\bibitem[AB]{am-br}
Amelotte, Steven; Briggs, Benjamin.
\href{https://arxiv.org/abs/2506.15457}{
\emph{Homotopy types of moment-angle complexes associated to almost linear resolutions}}. 
Preprint (2025); arXiv:2506.15457.

\bibitem[BM]{bo-me06}
Bosio, Fr\'ed\'eric; Meersseman, Laurent.
\href{https://doi.org/10.1007/s11511-006-0008-2}{
\emph{Real quadrics in $\mathbf C^n$, complex manifolds and convex polytopes.}} Acta Math. 197 (2006), no.~1, 53--127.

\bibitem[BP]{BP} 
Buchstaber, Victor; Panov, Taras.
\emph{Toric Topology.}
Math. Surveys Monogr.,~204,
Amer. Math. Soc., Providence, RI, 2015.

\bibitem[CFW]{c-f-w20}
Chen, Liman; Fan, Feifei; Wang, Xiangjun.
\emph{The topology of moment-angle manifolds---on a conjecture of S.~Gitler and S.~L\'opez de Medrano}.
Sci. China Math.~63 (2020), no.~10, 2079--2088.

\bibitem[FCMW]{FCMW}
Fan, Feifei; Chen, Liman; Ma, Jun; Wang, Xiangjun.
\href{https://projecteuclid.org/journals/osaka-journal-of-mathematics/volume-53/issue-1/Moment-angle-manifolds-and-connected-sums-of-sphere-products/ojm/1455892625.full}{\textit{Moment-angle manifolds and connected sums of sphere products}}.
Osaka J. Math.~53 (2016), no.~1, 31--45. 

\bibitem[FW]{FW21}
Fan, Feifei; Wang, Xiangjun.
\href{https://doi.org/10.1007/s11425-020-1889-8}{\textit{Moment-angle manifolds and connected sums of
simplicial spheres}}.
Sci. China Math.~64 (2021), no.~12, 2743--2758.

\bibitem[FG]{FG}
Fulkerson, Delbert; Gross, Oliver.
\href{http://dx.doi.org/10.2140/pjm.1965.15.835}{\textit{Incidence  matrices  and  interval  graphs}}.
Pacific J. Math~15 (1965), no.~3, 835--855.

\bibitem[GL]{gi-lo13}
Gitler, Samuel; L\'opez de Medrano, Santiago.
\href{https://doi.org/10.2140/gt.2013.17.1497}
{\emph{Intersections of quadrics, moment-angle manifolds and connected sums}}. Geom. Topol.~17 (2013), no.~3, 1497--1534.

\bibitem[I]{iriy18}
Iriye, Kouyemon. \emph{On the moment-angle manifold constructed by Fan, Chen, Ma and Wang.} Osaka J. Math.~55 (2018), no.~4, 587--593.

\bibitem[IK]{ir-ki17}
Iriye, Kouyemon; Kishimoto, Daisuke. 
\href{https://doi.org/10.1215/21562261-2017-0038}
{\emph{Fat-wedge filtration and decomposition of polyhedral products}}. Kyoto J. Math.~59 (2019), no.~1, 1--51.

\bibitem[K]{kovy}
Kovyrshina, Victoria. \emph{A new family of four-dimensional polytopes that yield connected sums of products of spheres.} In preparation.

\bibitem[KP]{kov-pa}
Kovyrshina, Victoria; Panov, Taras.
\href{https://doi.org/10.1134/S0001434626600316}
{\emph{Moment-angle manifolds corresponding to three-dimensional simplicial spheres, chordality and connected sums of products of spheres.}} 
Matem. zametki~119 (2026), no.~1, 65--76 (Russian); Math. Notes~119 (2026), no.~1, 70--81 (English translation).

\bibitem[Lo]{lope89}
L\'opez de Medrano, Santiago. 
\emph{Topology of the intersection of quadrics in ${\mathbb R}^n$.} In: Algebraic topology (Arcata, CA, 1986); Lecture Notes in Math.~1370, Springer, Berlin, 1989, pp.~280--292.

\bibitem[M]{mcga79}
McGavran, Dennis.
\href{https://doi.org/10.2307/1998691}
{\emph{Adjacent connected sums and torus actions.}} Trans. Amer. Math. Soc.~251 (1979), 235--254.

\bibitem[MT]{me-th}
Membrillo Solis, Amaranta; Theriault, Stephen.
\href{https://arxiv.org/abs/2605.01396}{
\emph{Moment-angle manifolds associated to neighbourly triangulations of spheres}}.
Preprint (2026); arXiv:2605.01396.

\bibitem[P]{pa-25}
Panov, Taras.
\href{https://doi.org/10.1112/blms.70122}
{\emph{Exponential actions defined by vector configurations, Gale duality, and moment?angle manifolds.}} Bull. Lond. Math. Soc.~57 (2025), no.~9, 2571--2629.
\end{thebibliography}
\end{document}